\documentclass[10pt]{amsart}
\usepackage[utf8]{inputenc}
\usepackage{amsmath, latexsym, amsfonts, amssymb, amsthm,enumitem}
\usepackage{graphicx, color,dsfont,epsfig,caption,wrapfig,subfig}
\usepackage{stmaryrd}
\usepackage{tikz}
\usepackage{bbm}
\usepackage{mathtools}

\usepackage{datetime}

\usepackage{hyperref}

 \hypersetup{
     linktoc=page,
     linkcolor=red,          
     citecolor=blue,        
     filecolor=blue,      
     urlcolor=cyan,
     colorlinks=true           
 }

\newcommand{\1}{{\ensuremath{\mathbbm{1}}}}

\newcommand{\fb}{{\ensuremath{\frac{\gb^2}{2}}}}

\def\y{\mathbf{y}}
\def\x{\mathbf{x}}

\newcommand{\la}{\lambda}

\numberwithin{equation}{section}
\newtheorem{theorem}{Theorem}[section]
\newtheorem{lemma}[theorem]{Lemma}
\newtheorem{proposition}[theorem]{Proposition}
\newtheorem{definition}[theorem]{Definition}
\newtheorem{assumption}[theorem]{Assumption}
\theoremstyle{remark}
\newtheorem{remark}[theorem]{Remark}
\theoremstyle{plain}

\DeclareMathSymbol{\geqslant}{\mathalpha}{AMSa}{"3E} \renewcommand{\geq}{\;\geqslant\;}                   \newcommand{\dd}{\text{\rm d}}             \newcommand{\R}{\mathbb{R}}
\newcommand{\N}{\mathbb{N}}
\newcommand{\Q}{\mathbb{Q}}
\newcommand{\E}{\mathds{E}}
\renewcommand{\P}{\mathds{P}}
\newcommand{\bx}{\mathbf{x}}

\usepackage{xparse}

\newcommand{\gb}{\beta}

\begin{document}

\title{Probabilistic ultraviolet renormalization  of the two-dimensional Sine-Gordon model}
 
 \author{Benjamin Bonnefont}
\author{Hermanni Rajamäki}
\author{Vincent Vargas}
\thanks{All authors: Université de Genève, Section de mathématiques et département de physique théorique, 1211 Geneva, Switzerland.}
\thanks{E-mails: \href{mailto:benjamin.bonnefont@unige.ch}{\texttt{benjamin.bonnefont@unige.ch}}, \href{mailto:hermanni.rajamaki@unige.ch}{\texttt{hermanni.rajamaki@unige.ch}}, \href{mailto:vincent.vargas@unige.ch}{\texttt{vincent.vargas@unige.ch}}.}



\begin{abstract}
We give a probabilistic approach to the ultraviolet renormalization of the two-dimensional Sine-Gordon model. The model is formally described by the density \(\exp(\alpha\int \mathopen{:} \cos(\beta X) \mathclose{:} )\), where \(X\) is a log-correlated Gaussian field. Our construction applies on bounded convex domains for every particle activity \(\alpha\in\mathbb R\) and throughout the subcritical range \(\beta^2\in[0,4)\), below the Berezinskii--\allowbreak Kosterlitz--\allowbreak Thouless threshold. We apply the martingale method introduced by Lacoin, Rhodes and Vargas to two dimensions by using refined estimates based on charge conjugation symmetries. Our result extends existing constructions by treating the full subcritical regime for a broad class of covariances.
\end{abstract}

\maketitle
\section{Introduction}
The two-dimensional Sine-Gordon model is a Euclidean field
theory formally defined for particle activity $\alpha\in \R$ and coupling parameter $\beta$ by the measure
\[
\exp\left(
\alpha\int_\Lambda\mathopen{:}\cos(\gb X(x))\mathclose{:}\mu(\dd x)
\right)\dd\P.
\]
$X$ is a centred Gaussian field with law $\P$ on an open bounded convex set $\Lambda\subset\R^2$ with covariance
$K(x,y)=-\ln|x-y|+O(1)$ and $\mu(\dd x)=g(x)\dd x$, where $g$ is some nonnegative and bounded function. The Wick ordering $\mathopen{:}\cdot\mathclose{:}$ is formally given by
\[
 \mathopen{:}\cos(\gb X(x))\mathclose{:}
 =e^{\fb K(x,x)}\cos(\gb X(x)).
\]
Here the diagonal covariance $K(x,x)$ is understood with an
implicit regularization, which will be specified below. The ultraviolet problem concerns removing this regularization. In contrast, the infrared problem relates to extending the domain $\Lambda$ to the whole plane. In this paper, we address the ultraviolet problem and keep $\Lambda$ fixed.

In the context of constructive quantum field theory, the Sine-Gordon model is of particular interest because of its renormalization structure: as $\beta^2$ increases, it undergoes a countable sequence of phase transitions. To describe this structure in more detail, we consider the (non-renormalized) partition function $\bar Z_{\alpha,\beta}(1)$, defined as the Gaussian expectation of the above density. Expanding the exponential, writing each cosine as a sum
of two complex exponentials and computing the expectation gives formally
\begin{equation*}
\begin{aligned}
 \bar Z_{\alpha,\beta}(1)
 &=\E\left[\exp\left(\alpha\int_\Lambda\mathopen{:}\cos(\gb X(x))\mathclose{:}\mu(\dd x)\right)\right]\\
 &=\sum_{n=0}^{\infty}\frac{\alpha^n}{2^n n!}
 \sum_{(\lambda_1,\ldots,\lambda_n)\in\{-1,1\}^n}
 \int_{\Lambda^n}
 \exp\left(
 -\gb^2\sum_{1\le k<l\le n}
 \lambda_k\lambda_l K(x_k,x_l)
 \right)
 \prod_{k=1}^n\mu(\dd x_k).
\end{aligned}
\end{equation*}
For $\alpha>0$, this is the partition function of a gas of
particles with charges $\lambda_k=\pm1$. The function is well-defined in the case $\gb^2\in[0,2)$ as the Wick ordering removes the interaction of each
charge with itself. The construction of the measure in this regime was established in the early results of the Sine-Gordon theory~\cite{Frohlich1975}. In the range $\gb^2\in[2,4)$ the partition function diverges as colliding opposite charges make the integrals infinite.

The first phase transition occurs at $\gb^2=2$: the attraction between two opposite charges forming a dipole becomes strong enough to produce a divergence in the integrals. The next transition, at $\gb^2=3$, occurs when the attraction between two dipoles, a quadrupole, becomes strong enough to produce additional divergences. Further multipoles create similar new divergences at each phase transition, the thresholds given by the formula
\begin{equation*}
	\beta_n^2=4-\frac{2}{n},\quad n\ge 1.
\end{equation*}
These successive multipole thresholds accumulate at the
Berezinskii--\allowbreak Kosterlitz--\allowbreak Thouless threshold~\cite{KosterlitzThouless1973}, which in our normalization equals $\beta^2=4$. Above this threshold, the theory is non-renormalizable. Below this threshold, however, only finitely
many types of multipoles give rise to new divergences and the theory is super-renormalizable.

Following~\cite{LacoinRhodesVargas2023}, a natural way to tame these divergences is to use the cumulants of $\int_\Lambda\mathopen{:}\cos(\gb X(x))\mathclose{:}\mu(\dd x)$ as counterterms, since these are field independent and their divergences follow the same pattern as in the multipole picture. Letting the $i$th cumulant be
$\mathcal C_i$, the renormalized partition function is then written
formally as
\begin{equation*}
Z_{\alpha,\beta}(1)
 =\E\left[\exp\left(\alpha\int_\Lambda\mathopen{:}\cos(\gb X(x))\mathclose{:}\mu(\dd x)\right)\right]
 \exp\left(-\sum_{i=1}^{2n-1}\frac{\alpha^i}{i!}\mathcal C_i\right),
\end{equation*}
where $n\ge1$ is the least integer such that $\gb^2<\gb_n^2$. The Sine-Gordon measure, after normalizing, can be written as
\[
\frac{1}{Z_{\alpha,\beta}(1)}\exp\left(
\alpha\int_\Lambda\mathopen{:}\cos(\gb X(x))\mathclose{:}\mu(\dd x)
\right)\exp\left(-\sum_{i=1}^{2n-1}\frac{\alpha^i}{i!}\mathcal C_i\right)\dd\P.
\]

We give a martingale construction for this probability measure throughout the full
subcritical range $\gb^2\in[0,4)$, for every particle
activity $\alpha\in\R$.
The construction applies on bounded convex domains
to a broad class of log-correlated Gaussian fields
satisfying the white noise decomposition stated below.
To our knowledge, previous constructions of the Sine-Gordon model beyond the first threshold
$\gb^2=2$ have been confined to specific massive or
massless Gaussian free field covariances, including the ones reviewed below.
Our method is not restricted to Gaussian free fields, nor do we require the translation invariance exploited
by the renormalization group constructions discussed below. For $\gb^2\in[0,2)$, constructions for general
log-correlated Gaussian fields follow from
imaginary multiplicative chaos~\cite{JunnilaSaksmanWebb2020}. All parameter ranges below are expressed in the normalization $K(x,y)=-\ln|x-y|+O(1)$.

Starting from renormalization group methods,
\cite{BenfattoGallavottiNicolo1982,Nicolo1983,NicoloRennSteinmann1986} provided bounds on the renormalized partition function, ultimately throughout the full subcritical regime. The approach was further developed in~\cite{BrydgesKennedy1987,DimockHurd1993,DimockHurd2000,Pelaic2025}\footnote{The article~\cite{DimockHurd1993} contains an error that the original authors correct in their follow-up~\cite{DimockHurd2000}. For a discussion of gaps in the arguments of
	~\cite{DimockHurd1993,DimockHurd2000}, see
	~\cite[Section~1.1]{Pelaic2025}.}. In particular, \cite{Pelaic2025} constructs the massive Sine-Gordon
measure on the torus throughout the subcritical range. At the critical value $\beta^2=4$, perturbative
	renormalizability was established in
~\cite{NicoloPerfetti1989}.

Probabilistic approaches have produced constructions in other settings, including ones without translational invariance, with the drawback of lacking the full subcritical range. For $\gb^2\in[0,2)$, stochastic quantization has been used to construct the massive
measure, both in finite and infinite volume~\cite{GubinelliHairerOhZine2025}, while~\cite{TangXu2025} used quadratic backward
stochastic differential equations 
to construct the measure in the case of  the Dirichlet Gaussian free field.
Beyond this regime, a continuous renormalization group formulation
of~\cite{BrydgesKennedy1987}
was used in~\cite{BauerschmidtWebb2024} to construct the massive
and massless models in finite volume for $\beta^2\in[0,3)$. This expansion was developed further in~\cite{Vihko2026},
which expands these results to higher dimensions. A construction of the massive model on the torus for
$\beta^2\in[0,3)$, together with a coupling to the Gaussian free
field, was obtained in~\cite{BauerschmidtHofstetter2022}. The forward-backward stochastic differential equation approach of~\cite{GubinelliMeyer2024} yields constructions of the massive
	model for $\beta^2\in[0,3)$ in both finite and infinite volume, and establishes further properties for the measure.
	In finite volume,~\cite{Meyer2026} extends the construction to
	$\beta^2\in[0,24/7)$ and also develops perturbative renormalization
	throughout the full subcritical regime.
	
Local well-posedness of the dynamical Sine-Gordon equation
	on the torus was established in~\cite{HairerShen2016}
	and extended to the full subcritical regime in
~\cite{ChandraHairerShen2018}.
	Global well-posedness beyond $\beta^2=2$ was obtained in
~\cite{ChandraFeltesWeber2026}, and throughout the range
	$\beta^2\in[0,3)$ in~\cite{BringmannCao2026}.
	Furthermore, \cite{ChandraFeltesWeber2026} proved the existence
	of an invariant measure for the massive dynamics in a regime
	extending beyond $\beta^2=2$.

Our starting point is the martingale method introduced in
~\cite{LacoinRhodesVargas2023}. In that work, Lacoin, Rhodes and Vargas construct the boundary (dimension one) Sine-Gordon theory for covariances satisfying the one-dimensional analogue of the white noise decomposition below. We extend this construction to two dimensions by exploiting charge conjugation arguments, which are closely related to those used in~\cite{ChandraHairerShen2018} and are also used in the analysis of the effective potential in ~\cite{Meyer2026}. In a nutshell, we symmetrise the integrands, arising from our martingale representations, by exchanging the coordinates of opposite charges. For instance, in the case of a dipole with charges at $x,y\in\Lambda$, this symmetrisation produces, at a fixed scale, a quadratic vanishing factor in the accompanying terms, improving local singularities of the form $|x-y|^{-\gb^2}$ to $|x-y|^{2-\gb^2}$, the latter being integrable throughout the full subcritical regime $\gb^2\in[0,4)$. A concrete example of these cancellations is given below, starting from Eq.~\eqref{prototype-contribution}.

\section{Setup and main results}

To regularize the formal model described above, we introduce the cutoff $t$, and consider the centred Gaussian field $(X_t(x))_{x\in \Lambda,t\ge0}$, where $\Lambda\subset\R^2$ is a bounded open convex set, $\mu(\dd x)=g(x)\dd x$ with $g\ge0$ bounded and $\mu(\Lambda)>0$. The covariance function is
\begin{equation*}
	\E[X_t(x)X_s(y)]=\int_0^{t\wedge s}Q_u(x,y)\dd u\eqqcolon K_{t\wedge s}(x,y),
\end{equation*}
where the kernel $Q_u(x,y)$ is symmetric, positive semidefinite for every $u\ge0$, and satisfies the following:

\begin{assumption}[Smooth white noise decomposition]\ \label{ass:white-noise}
\begin{enumerate}
 \item For every $u\ge0$, $Q_u\in C^2(\Lambda\times\Lambda)$.
 There exists a constant $C<\infty$ such that, for every
 $u\ge0$ and $x,y\in\Lambda$,
 \begin{equation}\begin{split}\label{laderive}
    |Q_u(x,y)| &\le C  (1+e^u |x-y|)^{-3},\\
    |\nabla_x Q_u(x,y)| &\le C e^{u}(1+ e^u|x-y|)^{-3},\\
       \|\nabla_x \nabla_y Q_u(x,y)\|_{\mathrm{op}} &\le C e^{2u}(1+e^u|x-y|)^{-3}.
\end{split}
 \end{equation}

  \item
\begin{equation}\label{log-covariance-bound}
\sup_{t\geq 0} \sup_{x,y\in \Lambda}|K_t(x,y)+\ln (|x-y|\vee e^{-t} )|<C.
\end{equation}
\end{enumerate}
\end{assumption}

\begin{remark}
	The kernels
	\[
	Q_u(x,y)=\exp\left(-\frac{e^{2u}|x-y|^2}{2}\right)
	\quad\text{and}\quad
	Q_u(x,y)=\exp\left(-\frac{m^2e^{-2u}}{2}
	-\frac{e^{2u}|x-y|^2}{2}\right),
	\]
	with $m>0$, are prototypical examples satisfying Assumption \ref{ass:white-noise} on any bounded open set $\Lambda$. They define translation-invariant log-correlated fields.
	
	The Gaussian heat kernel is also the basis for the construction of two-dimensional Gaussian free fields, normalized here so that their covariance has logarithmic singularity $-\log|x-y|$. More generally, suitable heat-kernel decompositions satisfy Assumption \ref{ass:white-noise} for the massive Gaussian free field on $\mathbb R^2$ restricted to $\Lambda$, as well as for the Dirichlet Gaussian free field, with or without mass, on a bounded smooth domain $D\subset\mathbb R^2$, restricted to $\Lambda$ with $\overline{\Lambda}\subset D$. The Dirichlet examples are not translation invariant. The Dirichlet Gaussian free field of $\Lambda$ itself is not covered by this assumption, since \eqref{log-covariance-bound} fails near $\partial\Lambda$.
\end{remark}

We use the jointly continuous version of this Gaussian field. Let $(\mathcal{F}_t)_{t\ge0}$ be the filtration generated by the field, and set $\mathcal{F}_\infty\coloneqq\sigma(\bigcup_{t\ge0}\mathcal{F}_t)$. We define the martingale
\begin{equation*}
 M_t\coloneqq\int_\Lambda\cos(\gb X_t(x))e^{\fb K_t(x,x)}\mu(\dd x).
\end{equation*}
The Sine-Gordon model is then defined by exponentiating this martingale. The multipole picture described in the introduction implies we have to subtract counterterms, the cumulants $\mathcal{C}_{i}(M_t)$ of $M_t$. More precisely, we study the regularized partition function with insertion $F$
\begin{equation}\label{partition-function-definition}
	Z_{\alpha,\beta,t}(F)\coloneqq
	\E[e^{\alpha M_t}F]e^{-\sum_{i=1}^{2n-1}\frac{\alpha^{i}}{i!}\mathcal{C}^{(\beta)}_{i}(M_t)},
\end{equation}
where $n\in\N$ is chosen such that $4-\frac{2}{n-1}\le\gb^2<4-\frac{2}{n}$, (for $n=1$, $0\le \gb^2<2$), $\alpha\in\R$  and $F$ is a bounded $\mathcal{F}_s$-measurable function for $0\le s\le t$. Our main result is then the following:
\begin{theorem}\label{maintheorem}
	Let $\gb^2<4$ and $\alpha\in\R$. Under Assumption \ref{ass:white-noise}, for $F$ bounded and $\mathcal{F}_s$-measurable for some $s\ge0$, the following limit exists
	\begin{equation*}
		Z_{\alpha,\beta}(F)\coloneqq\lim_{t\to\infty}Z_{\alpha,\beta,t}(F).
	\end{equation*}
Then
\begin{equation*}
	F\mapsto\frac{Z_{\alpha,\beta}(F)}{Z_{\alpha,\beta}(1)}
\end{equation*}
extends to all bounded $\mathcal{F}_{\infty}$-measurable $F$, and defines a probability measure on $\mathcal{F}_\infty$.
\end{theorem}
The construction of the correlation functions will be the subject of a forthcoming paper. We expect that the methods of our current paper will apply.

\subsection{Organization of the paper}
In Section 3, we will recall the necessary ingredients of \cite{LacoinRhodesVargas2023}, in particular  their representation of the partition function. We will also introduce a class of test functions that
allows us to exploit the cancellations from charge conjugations. In Section 4, using these test functions, we prove an estimate on a family of functions $F^{(i,p)}$ that are central to our representation of the partition function. The cancellations due to the charge conjugations happen precisely for these functions. Section 5 establishes auxiliary results. Section 6 uses these results and the estimate from Section 4 to prove the main theorem.

\section{Representation of the partition function}

\subsection*{Notation and conventions}
For the remainder of the paper, we write $\R_+=[0,\infty)$ and $\N=\{1,2,\ldots\}$.
For a finite
set $A$, $|A|$ denotes its cardinality. For real numbers $a,b$, we use
\[
 a\wedge b\coloneqq\min\{a,b\},\qquad
 a\vee b\coloneqq\max\{a,b\},\qquad
 a_+\coloneqq a\vee0.
\]
We write $|\cdot|$ for the Euclidean norm and $\| \cdot \|_{\mathrm{op}}$ for the operator norm induced by the Euclidean norm. For a bounded random variable $F$, $\|F\|_\infty$ is the essential
supremum of $|F|$ with respect to $\P$. The indicator of an event $E$ is $\1_E$ and
$\langle M,N\rangle$ is the quadratic covariation of continuous
semimartingales $M,N$, with $\langle M\rangle\coloneqq\langle M,M\rangle$.

In estimates, $C$, with or without subscripts, denotes a finite positive
deterministic constant whose value may change from line to line. Constants may also
depend on the fixed parameters and integer indices in the statement,
but are independent of the time variables and spatial configurations
unless such dependence is indicated, as in $C_s$. Constants used
in the definition of a function are fixed for that definition.

\subsection{Representation of the partition function}

To prove Theorem \ref{maintheorem} we will recall the representation of Eq.~\eqref{partition-function-definition} introduced in \cite[Lemma~3.1]{LacoinRhodesVargas2023}. For the martingale $M_t$ we define recursively the processes $(A^{(i)}_t)_{t\ge0}$ and $(M^{(i,t)}_s)_{t\ge s\ge0}$ for each $i\ge1$; the first cases are
\[
	A^{(1)}_t=M_t,\quad M^{(1,t)}_s=M_s-\mu(\Lambda),
\]
and for $i\ge2$
\begin{equation}\label{A-representation}
A_t^{(i)}=\frac{1}{2}\sum_{j=1}^{i-1}\langle M^{(j,t)},M^{(i-j,t)}\rangle_t, \quad M_s^{(i,t)}=\E[A_t^{(i)}\mid\mathcal F_s]-\E[A_t^{(i)}].
\end{equation}
Then, following \cite{LacoinRhodesVargas2023}, the cumulants of $M_t$ are given by
\begin{equation}\label{cumulant-identity}
	\mathcal{C}_{i}^{(\beta)}(M_t)=i!\E[A^{(i)}_t],
\end{equation}
and the partition function $Z_{\alpha,\beta,t}$ has the representation
\begin{equation}\label{partition-expansion}
	Z_{\alpha,\beta,t}(F)=\E\left[Fe^{{N_t^{(2n-1,\alpha,t)}-\frac{1}{2}\left\langle N^{(2n-1,\alpha,t)}\right\rangle_t+Q_{t,t}^{(2n-1,\alpha)}}}\right],
\end{equation}
where
\begin{equation}\label{remainder-definition}
	N_s^{(2n-1,\alpha,t)}=\sum_{i=1}^{2n-1}\alpha^iM_s^{(i,t)},\quad Q_{s,t}^{(2n-1,\alpha)}=\frac{1}{2}\sum_{i=2n}^{4n-2}\alpha^i\sum_{k=i+1-2n}^{2n-1}\langle M^{(k,t)},M^{(i-k,t)}\rangle_s.
\end{equation}
The advantage of this representation will be the martingale $(N_s^{(2n-1,\alpha,t)})_{0\le s\le t}$, which allows us to consider the limit $t\to\infty$ for each scale $s\ge0$. We will now introduce the explicit formulas for the quantities above, all of which are derived in \cite{LacoinRhodesVargas2023}.

First, let $\x=(x_1,\ldots,x_i)\in \Lambda^i$ be a charge configuration. For $0\le p\le \lfloor i/2\rfloor$ negative charges we define
\[
\lambda_{k}^{(p)}\coloneqq\begin{cases}
	-1,\quad k\le p,\\
	+1,\quad k>p.
\end{cases}
\]
We will also write
$Y_p(\x,u)\coloneqq\sum_{k=1}^i\lambda_k^{(p)}X_u(x_k)$ and
$\mu(\dd\x)\coloneqq\prod_{k=1}^i\mu(\dd x_k)$.
For $0\le s\le u$, define the regularised potentials
\begin{align*}
	U_{i,p}(\x,u)&\coloneqq2\sum_{1\le k<l\le i}\lambda^{(p)}_k\lambda^{(p)}_lK_u(x_k,x_l),\\
	U_{i,p}(\x,s,u)&\coloneqq U_{i,p}(\x,u)-U_{i,p}(\x,s).
\end{align*}
The formulas below follow the recursion for $A_t^{(i)}$ and
$M_s^{(i,t)}$. For $i\ge2$ and $1\le j\le i-1$, the brackets in
the definition of $A_t^{(i)}$ take the form
\begin{equation}\label{A-bracket}
 \langle M^{(i-j,t)},M^{(j,t)}\rangle_s
 =\sum_{p=0}^{\lfloor i/2 \rfloor}\int^s_0 \int_{\Lambda^i} F^{(i,j,p)}(\bx,u,t)e^{\sum_{k=1}^i\frac{\gb^2}{2}  K_u(x_k,x_k)}  \cos\left( \gb  Y_{p}(\x,u)\right) \mu(\dd \bx) \, \dd u,
\end{equation}
where $F^{(i,j,p)}$ are deterministic functions. The index $i$ counts
the coordinates and $p$ counts the negative charges.
Starting from $M_s^{(1,t)}=M_s-\mu(\Lambda)$, its quadratic variation gives
the initial cases
\begin{equation}\label{initial-split-kernels}
	F^{(2,1,1)}(x_1,x_2,u,t)=-F^{(2,1,0)}(x_1,x_2,u,t)=\fb Q_u(x_1,x_2).
\end{equation}
Taking half the sum over $j$ in Eq.~\eqref{A-bracket} at time $t$
gives $A_t^{(i)}$. Then Eq.~\eqref{cumulant-identity} gives
\begin{equation}\label{cumulantformula}
	\mathcal{C}_i^{(\beta)}(M_t)=i!\sum_{p=0}^{\lfloor i/2\rfloor}\int_0^t\int_{\Lambda^i}F^{(i,p)}(\x,u,t)e^{-\fb U_{i,p}(\x,u)}\mu(\dd\x)\dd u,
\end{equation}
where we set $F^{(i,p)}\coloneqq\frac12\sum_{j=1}^{i-1}F^{(i,j,p)}$. Taking the conditional expectation given
$\mathcal F_s$ of half the sum over $j$ in Eq.~\eqref{A-bracket}, evaluated at time $t$, and subtracting $\E[A_t^{(i)}]$, we also obtain, for
$0\le s\le t$ and $i\ge2$,
\begin{equation}\label{martingale-representation}
 M_s^{(i,t)}
 =\sum_{p=0}^{\lfloor i/2\rfloor}\int_0^t\int_{\Lambda^i}
 F^{(i,p)}(\x,u,t)e^{-\fb U_{i,p}(\x,u)}\\
 \times\left[
 \cos(\gb Y_p(\x,u\wedge s))
 e^{\fb\E[(Y_p(\x,u\wedge s))^2]}-1
 \right]\mu(\dd\x)\dd u.
\end{equation}
 To obtain the recursion  for $F^{(i,j,p)}$, we compute $\langle M^{(i,t)},M^{(j,t)}\rangle$ using Eq.~\eqref{martingale-representation}: Define
\begin{equation}\label{integrated-split-kernel}
 \mathcal{F}^{(i,j,p)}(\x,s,t)=\int_s^t F^{(i,j,p)}(\x,u,t)e^{-\fb U_{i,p}(\x,s,u)}\dd u,
\end{equation}
and
\begin{align}
 \mathcal F^{(i,p)}(\x,s,t)
 &=\frac12\sum_{k=1}^{i-1}\mathcal F^{(i,k,p)}(\x,s,t),
 \qquad i\ge2,\label{combined-integrated-kernel}\\
 \mathcal F^{(1,0)}(\x,s,t)&=1.\notag
\end{align}
Write $\x_1=(x_1,\ldots,x_i)$ and
$\x_2=(x_{i+1},\ldots,x_{i+j})$, and set
\begin{equation*}
 \mathcal{Q}_{i,j,p,q}(\x_1,\x_2,u)
 =\sum_{k=1}^{i}\sum_{l=i+1}^{i+j}
 \la^{(p)}_k\la^{(q)}_{l-i}Q_u(x_k,x_l).
\end{equation*}
The bracket computation gives the recursive identity
(cf.~\cite[Eqs.~(4.28) and~(4.32)]{LacoinRhodesVargas2023})
\begin{multline}\label{recursive-identity}
\sum_{r=0}^{\lfloor (i+j)/2 \rfloor} \int_{\Lambda^{i+j}} F^{(i+j,j,r)}(\bx,u,t)e^{\sum_{k=1}^{i+j}\frac{\gb^2}{2}K_u(x_k,x_k)}\cos\left(\gb Y_{r}(\x,u)\right)\mu(\dd\bx) \\
=\gb^2\sum_{p=0}^{\lfloor i/2\rfloor}\sum_{q=0}^{\lfloor j/2\rfloor}\int_{\Lambda^{i+j}}\mathcal{Q}_{i,j,p,q}(\x_1,\x_2,u)\mathcal{F}^{(i,p)}(\x_1,u,t)\mathcal{F}^{(j,q)}(\x_2,u,t)\\
\times e^{\sum_{k=1}^{i+j}\fb K_u(x_k,x_k)}\sin\left(\gb Y_{p}(\x_1,u)\right)\sin\left(\gb Y_{q}(\x_2,u)\right)\mu(\dd\x).
\end{multline}
The factors $\mathcal F^{(i,p)}$ and $\mathcal F^{(j,q)}$ are given by earlier steps
of the recursion. To choose a representative for each $F^{(i+j,j,r)}$, we write
\begin{multline}\label{sine-product}
 \sin\left(\gb Y_{p}(\x_1,u)\right)\sin\left(\gb Y_{q}(\x_2,u)\right)\\
=\frac{1}{2}\left[\cos\left(\gb\bigl(Y_{p}(\x_1,u)-Y_{q}(\x_2,u)\bigr)\right)-\cos\left(\gb\bigl(Y_{p}(\x_1,u)+Y_{q}(\x_2,u)\bigr)\right)\right].
\end{multline}
For the sum of the two $Y$ terms, let $T^+$ put the coordinates with
negative charges first. For their difference, define $T^-$ in the same
way, reversing all charges if necessary so that at most half are
negative. Since cosine is even, these coordinate permutations satisfy
\begin{align}
 &\cos\left(\gb Y_{r}(T^-\x,u)\right)=\cos\left(\gb\bigl(Y_{p}(\x_1,u)-Y_{q}(\x_2,u)\bigr)\right),\label{relabel-minus}\\
 &\cos\left(\gb Y_{p+q}(T^+\x,u)\right)=\cos\left(\gb\bigl(Y_{p}(\x_1,u)+Y_{q}(\x_2,u)\bigr)\right),\label{relabel-plus}
\end{align}
where $r=\min(p+j-q,q+i-p)$. For each pair $(p,q)$, Eqs.~\eqref{sine-product}--\eqref{relabel-plus}
then give a contribution to $F^{(i+j,j,p+q)}(T^+\x,u,t)$ with sign $-$
and a contribution to $F^{(i+j,j,r)}(T^-\x,u,t)$ with sign $+$.
Their common factor is
\begin{equation}\label{ijodd}
 \fb\mathcal{Q}_{i,j,p,q}(\x_1,\x_2,u)\mathcal{F}^{(i,p)}(\x_1,u,t)\mathcal{F}^{(j,q)}(\x_2,u,t).
	\end{equation}
After these changes of variables, we define each $F^{(i+j,j,r)}$
by adding the coefficients of $\cos(\gb Y_r(\x,u))$.
For fixed $(p,q)$, the two terms in Eq.~\eqref{sine-product}
contribute to the same function when $r=p+q$, which is equivalent
to $i=2p$ or $j=2q$. We can then combine them as follows.

Suppose $i=2p$. Let $T_1$ exchange $x_k$ and $x_{k+p}$ for
$k\le p$ and leave $\x_2$ unchanged. We may choose
$T^-=T^+\circ T_1$. Changing variables by $T_1$ in the term
with sign $+$ makes its cosine equal to that of the term with
sign $-$, while $\mathcal Q_{2p,j,p,q}$ changes sign.
Adding the two terms therefore gives the following contribution
to $F^{(i+j,j,p+q)}(T^+\x,u,t)$:
\begin{equation}\label{ieven}
 -\fb\mathcal{Q}_{2p,j,p,q}(\x_1,\x_2,u)\left(\mathcal{F}^{(2p,p)}(\x_1,u,t)+\mathcal{F}^{(2p,p)}(T_1\x_1,u,t)\right)\mathcal{F}^{(j,q)}(\x_2,u,t).
\end{equation}
The same argument applies when $j=2q$, with $T_2$ exchanging
$x_{i+l}$ and $x_{i+q+l}$ for $l\le q$ and leaving $\x_1$ unchanged.
If both $i=2p$ and $j=2q$, changing variables by $T_1$ and $T_2$
together leaves both $\mathcal Q$ and the cosine unchanged.
We may therefore average Eq.~\eqref{ieven} with its value at
$(T_1\x_1,T_2\x_2)$ without changing the integral.

 This gives
\begin{equation}\label{ijeven}
 -\frac{\gb^2}{4}\mathcal{Q}_{2p,2q,p,q}(\x_1,\x_2,u)\left(\mathcal{F}^{(2p,p)}(\x_1,u,t)+\mathcal{F}^{(2p,p)}(T_1\x_1,u,t)\right)\left(\mathcal{F}^{(2q,q)}(\x_2,u,t)+\mathcal{F}^{(2q,q)}(T_2\x_2,u,t)\right).
\end{equation}

\subsection{Class of test functions}
The remainder of the paper focuses on the following observation:
the recursive formulas~\eqref{recursive-identity}--\eqref{ijeven} express the cumulants in terms of lower-order $\mathcal{F}$ factors. For example, a contribution to $\int_{\Lambda^3}\mathcal{F}^{(3,1)}(\x,0,t)\mu(\dd\x)$ is of the form
\begin{align}
&\gb^2\int_{\Lambda^3}\int_0^t\phi(\x,u)\mathcal{F}^{(2,1)}(x_1,x_2,u,t)\dd u\,\mu(\dd\x),\label{prototype-contribution} \\
&\phi(\x,u)=(Q_u(x_1,x_3)-Q_u(x_2,x_3))e^{-\fb U_{3,1}(\x,u)}.\label{prototype}
\end{align}
The lower-order $\mathcal{F}^{(2,1)}$ has a singularity of order $|x_1-x_2|^{-\gb^2}$ as the cutoff $t$ is removed, which is locally integrable only when $\beta^2<2$. Letting $\tau$ be the charge conjugation that exchanges $x_1$ and $x_2$, we can symmetrise Eq.~\eqref{prototype-contribution} as
\begin{equation*}
 \frac{\gb^2}{2}\int_{\Lambda^3}\int_0^t(\phi(\x,u)+\phi(\tau\x,u))\mathcal{F}^{(2,1)}(x_1,x_2,u,t)\dd u\,\mu(\dd\x),
\end{equation*}
where
\begin{equation*}
\begin{split}
\phi(\x,u)+\phi(\tau\x,u)
={}&2\bigl(Q_u(x_1,x_3)-Q_u(x_2,x_3)\bigr)e^{\gb^2K_u(x_1,x_2)}\\
&\times\sinh\!\left(\gb^2\bigl(K_u(x_1,x_3)-K_u(x_2,x_3)\bigr)\right).
\end{split}
\end{equation*}
At fixed $u$, the difference $Q_u(x_1,x_3)-Q_u(x_2,x_3)$, and the $\sinh$ term each provide a factor of $|x_1-x_2|$, improving the local integrability to $\gb^2<4$. Following this idea, we will define a class of test functions, the function $(x_1,x_2,u)\mapsto\int_\Lambda\phi((x_1,x_2,x_3),u)\mu(\dd x_3)$ from Eq.~\eqref{prototype} being a prototype member of this class.

Before this, we introduce some notation: for a vector $\x=(x_1,\ldots,x_n)$, $n\ge2$, and integers $1\le a<b\le n$, let $\tau_{(a,b)}$ be the transposition that exchanges components $a$ and $b$.
\begin{equation*}
 \tau_{(a,b)}(x_1,\ldots,x_a,\ldots,x_b,\ldots,x_n)=(x_1,\ldots,x_b,\ldots,x_a,\ldots,x_n).
\end{equation*}
Let $P_n\coloneqq\{(a,b):1\le a<b\le n\}$. We define
\[
\Omega_n\coloneqq\{M\subseteq P_n:\{a,b\}\cap\{c,d\}=\emptyset
\text{ for all distinct }(a,b),(c,d)\in M\}.
\]
More generally, we will write $\Omega_\x$ when we take the set of labels from $\x$.
For $A\in\Omega_n$, write
\[
 D_A\coloneqq\prod_{a\in A}(1-\tau_a),
 \qquad \tau_A\coloneqq\prod_{a\in A}\tau_a,
 \qquad D_\emptyset=\tau_\emptyset=\mathrm{Id}.
\]
Here $(\tau_a f)(\x)=f(\tau_a\x)$, and for a single pair $a$ we abbreviate $D_{\{a\}}$ to $D_a=1-\tau_a$. Since the swaps in $A$ are disjoint, the products do not depend on their order.
 Empty products of scalar factors equal $1$; empty products of permutation or difference operators are the identity operator and leave their argument unchanged.  Relevant formulas for these operators are proved in Appendix~\ref{app:swap-algebra}.
\begin{definition}[class of test functions]\label{TestDef}
    Let $i\ge2$, $0\le p\le\lfloor i/2\rfloor$, and fix $\gb^2/4<\gamma\le1$. Then $\mathcal D_{i,p}^\gamma(\Lambda^i\times\R_+)$ is the class of measurable functions $(\x,u)\mapsto\phi_i(\x,u)$, with $\x\in \Lambda^i$, for which:
    \begin{enumerate}
        \item\label{test:envelope} there exists a measurable majorant $F_i(\x,u)=\sum_{k=1}^i f(x_k,u)$, where $f\ge0$ and $f(\cdot,u)\in L^1(\Lambda,\mu)$ for every $u$;
        \item for every $u\ge0$ and $M\in\Omega_\x$,
        \begin{equation}\label{testprop}
        \left|D_M
        \phi_i(\x,u)\right|
        \le F_i(\x,u)
        \prod_{(m_1,m_2)\in M}\bigl[1\wedge(e^u|x_{m_1}-x_{m_2}|)\bigr]^\gamma.
        \end{equation}
        \item\label{test:symmetry} for every $u\ge0$, $\phi_i(\cdot,u)$ is invariant under permutations of the first $p$ coordinates and, independently, of the last $i-p$ coordinates.
    \end{enumerate}
\end{definition}

The main proposition we are going to prove is the following:

\begin{proposition}\label{prop:pairing}
Let $\gb^2<4$, $i\ge2$, $0\le p\le\lfloor i/2\rfloor$, and $\phi_i\in\mathcal D_{i,p}^\gamma(\Lambda^i\times\R_+)$, with $\gb^2/4<\gamma\le1$ and majorant $F_i(\x,u)=\sum_{k=1}^i f(x_k,u)$. There exists a constant $C>0$, independent of $\phi_i$ and $f$, such that for every $t\ge u\ge0$,
\begin{equation}\label{propz1}
 \left|\int_{\Lambda^i}F^{(i,p)}(\x,u,t)\phi_i(\x,u)
       \mu(\dd\x)\right|
 \le Ce^{-2(i-1)u}\int_\Lambda f(x,u)\mu(\dd x),
\end{equation}
where $\x=(x_1,\ldots,x_i)$. Furthermore,
\begin{equation}\label{pairing-limit}
	\lim_{t\to \infty}\int_{\Lambda^i}F^{(i,p)}(\x,u,t)\phi_i(\x,u)
       \mu(\dd\x)
\end{equation}
exists. The same bound and limit assertion hold with $F^{(i,p)}$ replaced by $F^{(i,j,p)}$, for every $1\le j\le i-1$.
\end{proposition}
This proposition will be the main tool for us to prove Theorem \ref{maintheorem}. Furthermore, we can prove the following convergence statements for the cumulants. This is consistent with the multipole picture, in the sense that each even cumulant converges below its corresponding threshold.
\begin{theorem}\label{thm:cumulants}
Assume $\gb^2<4$. For every integer $n\ge1$,
$\mathcal C_{2n-1}(M_t)$ converges to a finite limit as
$t\to\infty$. Moreover, $\mathcal C_{2n}(M_t)$ converges
to a finite limit whenever $\gb^2<4-2/n=\beta_n^2$.
\end{theorem}

\begin{proof}
The first cumulant equals $\mu(\Lambda)$.
For $i\ge2$, Eq.~\eqref{cumulantformula} gives
\[
\frac{\mathcal C_i(M_t)}{i!}
=
\sum_{p=0}^{\lfloor i/2\rfloor}
\int_0^t\int_{\Lambda^i}
F^{(i,p)}(\x,u,t)e^{-\fb U_{i,p}(\x,u)}
\mu(\dd\x)\,\dd u.
\]
 Lemma~\ref{lem:energy-multiplier}, proved below, shows that $e^{-\fb U_{i,p}(\x,u)}$ belongs to the class of test functions with majorant $Ce^{\fb(i-\kappa_{i,p})u}$, where $\kappa_{i,p}=0$ when $i=2p$ and $1$ otherwise. Apply Proposition~\ref{prop:pairing}:\[
\left|
\int_{\Lambda^i}F^{(i,p)}(\x,u,t)
e^{-\fb U_{i,p}(\x,u)}\mu(\dd\x)
\right|
\le C
\begin{cases}
e^{-(2-\fb)(i-1)u},&i\ne2p,\\
e^{-[2(i-1)-\fb i]u},&i=2p.
\end{cases}
\]

For odd $i$, the first bound applies to every $p$ and is
integrable on $[0,\infty)$.
For $i=2n$, both bounds are integrable when
$\gb^2<4-2/n$.
Dominated convergence proves the finite limits.
\end{proof}

\section{Proof of Proposition \ref{prop:pairing}}
Before beginning the proof, we will first collect a few important estimates. First, the potential $U_{i,p}$ is related to the variance of the random variable $\sum_{k=1}^{i}\la^{(p)}_k(X_u(x_k)-X_s(x_k))$. In particular, using the nonnegativity of the variance and Eq.~\eqref{log-covariance-bound}, we have the lower bound
\begin{equation}\label{energy-crude}
	U_{i,p}(\x,s,u)\ge -i(u-s)-C.
\end{equation}
Lacoin, Rhodes and Vargas~\cite[Lemma~4.3]{LacoinRhodesVargas2023} refine the bound~\eqref{energy-crude} by looking at the distance between the negative and positive charges. In the case $i=2p$ the relevant quantity is the 1-Wasserstein distance
\begin{align}\label{matching-distance}
    m(\x)=\min_{\sigma\in S_p}\sum_{k=1}^p|x_k-x_{p+\sigma(k)}|,
\end{align}
where the minimum is taken over all permutations of $p$ elements. The bounds are explicitly as follows:
\begin{theorem}\label{thm:onsager}
 For all integers $i\ge2$ and $0\le p\le\lfloor i/2\rfloor$, there exists a constant $C$ such that for all $\x\in \Lambda^i$ and $0\le s\le u$
 \begin{enumerate}
     \item For $i=2p$,
     \begin{equation}\label{energy-onsager-neutral}
     U_{i,p}(\x,s,u)\ge-(i-1)(u-s)-(-\ln m(\x)-s)_+-C.
     \end{equation}
     \item Otherwise,
     \begin{equation}\label{energy-onsager-nonneutral}
     U_{i,p}(\x,s,u)\ge-(i-1)(u-s)- C.
     \end{equation}
 \end{enumerate}
\end{theorem}

We will use the following estimates repeatedly in the remainder of the paper; their proof is given in Appendix~\ref{app:Q-estimates}.

\begin{lemma}\label{lem:Q-estimates}
Let $i,j\ge1$, $0\le p\le\lfloor i/2\rfloor$ and
$0\le q\le\lfloor j/2\rfloor$ be integers. There exists a constant $C$
such that, for all $\x_1=(x_1,\ldots,x_i)\in \Lambda^i$,
$\x_2=(x_{i+1},\ldots,x_{i+j})\in \Lambda^j$ and $u\ge0$,
\begin{equation}\label{Q-basic}
 |\mathcal{Q}_{i,j,p,q}(\x_1,\x_2,u)|\le C\sum_{k=1}^i\sum_{l=i+1}^{i+j}\frac{1}{(1+e^u|x_k-x_l|)^{3}}\eqqcolon G_{i,j}(\x_1,\x_2,u).
\end{equation}
If $i=2p$, then
\begin{equation}\label{Q-one-neutral}
	|\mathcal{Q}_{2p,j,p,q}(\x_1,\x_2,u)|\le [1\wedge m(\x_1)e^u]G_{2p,j}(\x_1,\x_2,u).
\end{equation}
A similar bound holds when $j=2q$, with the factor
$[1\wedge m(\x_2)e^u]$. If both $i=2p$ and $j=2q$, then
\begin{equation}\label{Q-two-neutral}
	|\mathcal{Q}_{2p,2q,p,q}(\x_1,\x_2,u)|\le [1\wedge m(\x_1)e^u][1\wedge m(\x_2)e^u]G_{2p,2q}(\x_1,\x_2,u).
\end{equation}
\end{lemma}
Lastly, before proving Proposition~\ref{prop:pairing}, we show the following properties for our class of test functions:
\begin{lemma}\label{lem:general-energy-multiplier}\label{lem:energy-multiplier}
Fix $\gb^2<4$, $i\ge2$ and $0\le p\le\lfloor i/2\rfloor$, and let $\x\in \Lambda^i$. Let $\phi\in\mathcal{D}^\gamma_{i,p}(\Lambda^i\times\R_+)$ with majorant $F_i(\x,u)$. For fixed $u\ge0$, define $\psi(\x,t)\coloneqq e^{-\fb U_{i,p}(\x,u,t)}\phi(\x,u)$ for $t\ge u$, and set $\psi(\x,t)=0$ for $0\le t<u$. Then
\begin{enumerate}
	\item $\psi\in \mathcal{D}^\gamma_{i,p}$ with majorant $C e^{\fb(i-\kappa_{i,p})(t-u)}F_i(\x,u)$, where $\kappa_{i,p}=0$ if $i=2p$ and $1$ otherwise,
	\item if $i=2p$ and 
	\begin{align}\left|D_M\phi(\x,u)\right|
&\le F_{2p}(\x,u)\chi_M(\x,u)
\prod_{(m_1,m_2)\in M}\bigl[1\wedge(e^u|x_{m_1}-x_{m_2}|)\bigr]^\gamma, \label{unweighted-property}
 \\
\chi_M(\x,u)&=\begin{cases}
1,& |M|\ge2,\\
\bigl([1\wedge m(\x)e^u]+[1\wedge|x_{m_1}-x_{m_2}|e^u]\bigr)^\gamma,& M=\{m\},\\
[1\wedge m(\x)e^u]^{2\gamma},&M=\emptyset,
\end{cases}\label{chi-definition}
\end{align}
then $\psi\in\mathcal{D}_{i,p}^\gamma$ with majorant $Ce^{\fb(i-1)(t-u)}F_i(\x,u)$.
\end{enumerate}
\end{lemma}
\begin{proof}
We start from the first item: $U_{i,p}$ satisfies the symmetry condition in item~\ref{test:symmetry} of Definition~\ref{TestDef}. Then, we need to show that for all $t\ge u$
\begin{equation}\label{general-weighted-test}
\left|D_M\left[e^{-\fb U_{i,p}(\x,u,t)}\phi(\x,u)\right]\right|
\le C e^{\fb(i-\kappa_{i,p})(t-u)}F_i(\x,u)
\prod_{(m_1,m_2)\in M}\bigl[1\wedge(e^t|x_{m_1}-x_{m_2}|)\bigr]^\gamma.
\end{equation}
Set $A\coloneqq\{m\in M:|x_{m_1}-x_{m_2}|< e^{-t}\}$. The expansion~\eqref{app:partial-subset-expansion} gives
\begin{equation}\label{swap-expansion}
D_M
\left[e^{-\fb U_{i,p}(\x,u,t)}\phi(\x,u)\right]
=\sum_{B\subseteq M\setminus A}(-1)^{|B|}
\left.\left\{D_A
\left[e^{-\fb U_{i,p}(\x,u,t)}\phi(\x,u)\right]\right\}\right|_{\x=\y_B},
\end{equation}
where $\y_B=\tau_B\x$. Therefore it suffices to consider the case $M=A$. Using~\eqref{app:product-rule-second} on the left-hand side of Eq.~\eqref{swap-expansion} gives
\begin{equation}\label{energy-Leibniz}
\sum_{A\subseteq M}
\left[\tau_{A^c}D_Ae^{-\fb U_{i,p}(\x,u,t)}\right]
\left[D_{A^c}\phi(\x,u)\right].
\end{equation}
Here $A^c=M\setminus A$. We aim to prove $\left|D_Ae^{-\fb U_{i,p}(\x,u,t)}\right|\le C e^{-\fb U_{i,p}(\x,u,t)}\prod_{a\in A}|x_{a_1}-x_{a_2}|e^t$, since then the absolute value of Eq.~\eqref{energy-Leibniz} is bounded by
\begin{equation}\label{midpoint}
	C\sum_{A\subseteq M}\left(\tau_{A^c}e^{-\fb U_{i,p}(\x,u,t)}\right)\left(\prod_{a\in A}|x_{a_1}-x_{a_2}|e^t\right)\left|D_{A^c}\phi(\x,u)\right|
\end{equation}
and Eq.~\eqref{general-weighted-test} then follows by using Eqs.~\eqref{energy-crude} and~\eqref{energy-onsager-nonneutral} and $\phi\in\mathcal{D}^\gamma_{i,p}$. If $A=\emptyset $ the result is immediate. Otherwise, choose $\theta\in A$. By induction on $|A|$, using Eq.~\eqref{app:product-rule-first},
\begin{multline}\label{energy-difference-induction}
	D_Ae^{-\fb U_{i,p}(\x,u,t)}=D_{A\setminus\{\theta\}}\left(e^{-\fb U_{i,p}(\x,u,t)}\left[1-e^{\fb D_{\theta}U_{i,p}(\x,u,t)}\right]\right)\\
	=\sum_{B\subseteq A\setminus\{\theta\}}\left(D_B e^{-\fb U_{i,p}(\x,u,t)}\right)\left(\tau_B D_{(A\setminus\{\theta\})\setminus B}\left[1-e^{\fb D_{\theta}U_{i,p}(\x,u,t)}\right]\right).
\end{multline}
The first factor in the sum is controlled by the induction hypothesis. For the second, we observe the following: by expanding the definition of $U_{i,p}$ and applying Eq.~\eqref{laderive} we obtain for pairwise disjoint pairs $a=(a_1,a_2),b=(b_1,b_2)$ and $c=(c_1,c_2)$
\begin{align}
	|D_aU_{i,p}(\x,u,t)|&\le C|x_{a_1}-x_{a_2}|e^t,\label{energy-first-difference}\\
	|D_{\{a,b\}}U_{i,p}(\x,u,t)|&\le C|x_{a_1}-x_{a_2}||x_{b_1}-x_{b_2}|e^{2t},\label{energy-second-difference}\\
D_{\{a,b,c\}}U_{i,p}&=0.\label{energy-third-difference}
\end{align}
Eq.~\eqref{energy-first-difference} proves the case $|A|=1$:
\[
\left|D_ae^{-\fb U_{i,p}}\right|\le \fb |U_{i,p}-\tau_a U_{i,p}| e^{-\fb\min\left(U_{i,p},\tau_a U_{i,p}\right)}\le Ce^{-\fb U_{i,p}(\x,u,t)}|x_{a_1}-x_{a_2}|e^t,
\]
where in the last inequality we applied Eq.~\eqref{energy-first-difference} to both the first factor and the exponential. Next, for $|A|\ge 2$ by direct computation
\begin{equation*}
	D_b\left(1-e^{\fb D_{\theta}U_{i,p}}\right)=-e^{\fb D_{\theta}U_{i,p}}\left(1-e^{-\fb D_{\{\theta,b\}}U_{i,p}}\right).
\end{equation*}
The term in the bracket is invariant under a third swap $\tau_c$ per Eq.~\eqref{energy-third-difference}. Thus, by induction for all nonempty $C\subseteq A\setminus\{\theta\}$
\begin{equation*}
	D_C\left(1-e^{\fb D_{\theta}U_{i,p}}\right)=-e^{\fb D_{\theta}U_{i,p}}\prod_{c\in C}\left(1-e^{-\fb D_{\{\theta,c\}}U_{i,p}}\right).
\end{equation*}
For $C=\emptyset$, Eq.~\eqref{energy-first-difference} gives $|1-e^{\fb D_\theta U_{i,p}}|\le C|x_{\theta_1}-x_{\theta_2}|e^t$. For the terms in the product
\begin{equation*}
	\left|1-e^{-\fb D_{\{\theta,c\}}U_{i,p}}\right|\le \fb |D_{\{\theta,c\}}U_{i,p}|e^{\fb\left|D_{\{\theta,c\}}U_{i,p}\right|}\le C|x_{\theta_1}-x_{\theta_2}||x_{c_1}-x_{c_2}|e^{2t},
\end{equation*}
where in the last inequality we used Eq.~\eqref{energy-second-difference}. Since $|x_{\theta_1}-x_{\theta_2}|e^t<1$, its $|C|$th power is at most its first power for nonempty $C$. These estimates also hold after $\tau_B$, since the swaps are disjoint. Plugging them into Eq.~\eqref{energy-difference-induction} proves the induction.

Turning to the second item, we again assume $|x_{a_1}-x_{a_2}|e^t<1$ for all $a\in M$, and it then suffices to show that
\begin{equation}\label{neutral-weighted-test}
	|D_{M}\psi(\x,t)|\le CF_{2p}(\x,u)e^{\fb(2p-1)(t-u)}\prod_{m\in M}(|x_{m_1}-x_{m_2}|e^t)^\gamma.
\end{equation}
For $M=\emptyset$, Eq.~\eqref{energy-onsager-neutral} gives the bound for $m(\x)>0$
\[
 C F_{2p}(\x,u)e^{\fb(i-1)(t-u)}[1\wedge e^um(\x)]^{2\gamma-\fb}
 \le C F_{2p}(\x,u)e^{\fb(i-1)(t-u)};
\]
when $m(\x)=0$, Eq.~\eqref{unweighted-property} implies $\phi(\x,u)=0$. For $|M|\ge1$, we use Eq.~\eqref{midpoint} with Eq.~\eqref{unweighted-property}:
\begin{equation*}
	|D_M\psi(\x,t)|\le CF_{2p}(\x,u)\left(\sum_{A\subseteq M}\chi_{A^c}(\x,u)\tau_{A^c}\left(e^{-\fb U_{i,p}(\x,u,t)}\right)e^{-\gamma|A^c|(t-u)}\right)\prod_{m\in M}(|x_{m_1}-x_{m_2}|e^t)^\gamma.
\end{equation*}
To show Eq.~\eqref{neutral-weighted-test} it suffices to show that
\begin{equation}\label{energy-compensation}
 \chi_A(\x,u)e^{-\fb U_{i,p}(\tau_A \x,u,t)}e^{-\gamma|A|(t-u)}\le C e^{\fb(2p-1)(t-u)}
 \end{equation}
 for all $A\subseteq M$. To this end write $\y_A\coloneqq\tau_A\x=(y_1,...,y_{2p})$, and let $\sigma_\x$ minimize the expression in Eq.~\eqref{matching-distance} for $m(\x)$. Then
\begin{equation}
m(\y_A)\le\sum_{k=1}^p\bigl(|x_k-x_{p+\sigma_\x(k)}|+|y_k-x_k|+|y_{p+\sigma_\x(k)}-x_{p+\sigma_\x(k)}|\bigr)=m(\x)+\sum_{k=1}^i|y_k-x_k|.\label{matching-Lipschitz}
\end{equation}
The same argument applies with $\y_A$ and $\x$ interchanged. Only coordinates moved by $\tau_{A}$ contribute, so
\begin{equation}\label{short-matching-distance}
|m(\x)-m(\y_A)|\le\sum_{k=1}^i|y_k-x_k|
=2\sum_{m\in A}|x_{m_1}-x_{m_2}|\le2|M|e^{-t}.
\end{equation}
If $m(\x)\ge2m(\y_A)$, Eq.~\eqref{short-matching-distance} gives $m(\x)\le4|M|e^{-t}$; hence Eq.~\eqref{chi-definition} gives
\begin{equation}\label{chi-short-bound}
\chi_{A}(\x,u)\le C\begin{cases}
1,&|A|\ge2,\\
e^{-\gamma(t-u)},&|A|=1,\\
e^{-2\gamma(t-u)},&A=\emptyset.
\end{cases}
\end{equation}
Applying Eqs.~\eqref{energy-crude} and~\eqref{chi-short-bound}, we obtain
\begin{equation}\label{energy-small-matching}
e^{-\fb U_{i,p}(\y_A,u,t)}\chi_{A}(\x,u)e^{-\gamma|A|(t-u)}
\le Ce^{\fb i(t-u)}e^{-2\gamma(t-u)}.
\end{equation}
Since $2\gamma>\fb$, we prove Eq.~\eqref{energy-compensation} in this case.
If $m(\x)<2m(\y_A)$, then $1\wedge m(\x)e^u\le2(1\wedge m(\y_A)e^u)$.
When $m(\x)\ge e^{-t}$, Eqs.~\eqref{energy-onsager-neutral} and~\eqref{chi-definition} give
\begin{equation}\label{energy-comparable-matching}
\begin{split}
&e^{-\fb U_{i,p}(\y_A,u,t)}\chi_{A}(\x,u)e^{-\gamma|A|(t-u)}\\
&\quad\le Ce^{\fb(i-1)(t-u)}\begin{cases}
[1\wedge m(\x)e^u]^{-\fb}e^{-2\gamma(t-u)},&|A|\ge2,\\
e^{-\gamma(t-u)}\bigl([1\wedge m(\x)e^u]^{\gamma-\fb}
+e^{-\gamma(t-u)}[1\wedge m(\x)e^u]^{-\fb}\bigr),&|A|=1,\\
[1\wedge m(\x)e^u]^{2\gamma-\fb},&A=\emptyset.
\end{cases}
\end{split}
\end{equation}
In Eq.~\eqref{energy-comparable-matching}, for $A=\emptyset$, use $2\gamma>\fb$. For $|A|\ge2$ and the second summand when $|A|=1$, use $m(\x)\ge e^{-t}$. For the first summand in the case $|A|=1$, use the same inequality if $\fb>\gamma$; if $\fb\le\gamma$, the factor with exponent $\gamma-\fb$ is at most $1$.
When $m(\x)<e^{-t}$, use Eq.~\eqref{energy-crude} as in the derivation of Eq.~\eqref{energy-small-matching}.
\end{proof}

\begin{proof}[Proof of Proposition~\ref{prop:pairing}]
We start by proving Eq.~\eqref{propz1}. For $i=2$ and $p=0,1$, the left-hand side of Eq.~\eqref{propz1} is bounded, up to a fixed coefficient, by
\begin{equation*}
 \int_{\Lambda^2}|Q_u(x_1,x_2)|\left|\phi_2((x_1,x_2),u)\right|\mu(\dd\x)\le 2\int_\Lambda f(x,u)\left(\int_\Lambda|Q_u(x,y)|\mu(\dd y)\right)\mu(\dd x)
 \le Ce^{-2u}\int_\Lambda f(x,u)\mu(\dd x),
\end{equation*}
the first inequality coming from applying Eq.~\eqref{testprop} to $\phi_2$ with $M=\emptyset$, the last from Eq.~\eqref{laderive} and boundedness of the density of $\mu$.

Next, we prove the bound for each contribution in Eqs.~\eqref{ijodd}--\eqref{ijeven}; this proves it for every $F^{(i,j,p)}$ and hence for $F^{(i,p)}$. We consider the hardest case, Eq.~\eqref{ijeven}, with $i=2p$, $j=2q$ and $r=p+q$; the other cases follow the same strategy. Since $U_{i,p}$ is invariant under $T_1$, and $U_{j,q}$ under $T_2$, the contribution to be bounded, up to the fixed coefficient $-\gb^2/4$, is
\begin{multline}\label{neutral-child-product}
  \int_{\Lambda^{i+j}}\bigg[\mathcal{Q}_{i,j,p,q}(\bx_1,\bx_2,s)  \left(\int^t_s (F^{(i,p)}(\mathbf{x}_1,u,t)+F^{(i,p)}(T_1\mathbf{x}_1,u,t)) e^{ - \fb  U_{i,p}(\bx_1,s,u)} \dd u\right)\\
 \times \left(\int^t_s (F^{(j,q)}(\mathbf{x}_2,v,t)+F^{(j,q)}(T_2 \mathbf{x}_2,v,t))  e^{-\fb  U_{j,q}(\bx_2,s,v)}\dd v \right)\bigg]\phi_{i+j}(T^+\x,s)\mu(\dd\x),
\end{multline}
where $T^+$ is the permutation coming from Eqs.~\eqref{relabel-minus} and~\eqref{relabel-plus}. After expanding the time integrals in Eq.~\eqref{neutral-child-product}, the absolute value of the spatial integral at fixed $s\le v,h\le t$ can be written as 
\begin{multline*}
 \Bigg|\int_{\Lambda^i}F^{(i,p)}(\x_1,v,t)e^{-\fb U_{i,p}(\x_1,s,v)}
 \int_{\Lambda^j}F^{(j,q)}(\x_2,h,t)e^{-\fb U_{j,q}(\x_2,s,h)}
 \mathcal Q_{i,j,p,q}(\x_1,\x_2,s)\\
 \times(1-T_1)(1-T_2)\phi_{i+j}(T^+\x,s)
 \mu(\dd\x_2)\mu(\dd\x_1)\Bigg|.
\end{multline*}
Define
\begin{align*}
 \Phi(\x,s)&\coloneqq(1-T_1)(1-T_2)\phi_{i+j}(T^+\x,s).
\end{align*}
We will furthermore split
\[
 F_1(\x,s)=\sum_{k=1}^i f(x_k,s),\qquad
 F_2(\x,s)=\sum_{l=i+1}^{i+j}f(x_l,s),\qquad
 \Phi_a=\frac{F_a}{F_1+F_2}\Phi\quad(a=1,2),
\]
and on $\{F_1+F_2=0\}$, where $\Phi=0$, we set $\Phi_1=\Phi_2=0$. Write $F\coloneqq F_1+F_2$. We will prove that, for all $M\in\Omega_{\x_1}$ and $N\in\Omega_{\x_2}$,
\begin{multline}\label{jgoal-split}
 \left|D_{M\cup N}
 [\mathcal Q(\x_1,\x_2,s)\Phi_a(\x,s)]\right|\le CF_a(\x_a,s)G_{i,j}(\x_1,\x_2,s)
 \chi_M(\x_1,s)\chi_N(\x_2,s)\\
 \times\prod_{(m_1,m_2)\in M}[1\wedge e^s|x_{m_1}-x_{m_2}|]^\gamma
 \prod_{(b_1,b_2)\in N}[1\wedge e^s|x_{b_1}-x_{b_2}|]^\gamma.
\end{multline}
Then by Lemma~\ref{lem:energy-multiplier}, for fixed $s$ and $\x_1$, $(\x_2,h)\mapsto D_M\left(e^{-\fb U_{j,q}(\x_2,s,h)}\mathcal{Q}(\x_1,\x_2,s)\Phi_1(\x,s)\right)$ belongs to $\mathcal{D}_{j,q}^\gamma(\Lambda^j\times\R_+)$ with majorant $C_M(\x_1,s)G(\x_2,s)e^{\fb(j-1)(h-s)}$, where $$C_M(\x_1,s)=CF_1(\x_1,s)\chi_M(\x_1,s)\prod_{(m_1,m_2)\in M}[1\wedge e^s|x_{m_1}-x_{m_2}|]^\gamma,$$ $G(\x_2)=\sum_{k=i+1}^{i+j}g(x_k)$ and $g(x_k)=\sum_{l=1}^i(1+e^s|x_l-x_k|)^{-3}$. Since $\int_{\Lambda}g(x)\mu(\dd x)\le Ce^{-2s}$, by the induction hypothesis on $j$, we have
\begin{multline*}
	\left|D_M\left(\int_{\Lambda^j}F^{(j,q)}(\x_2,h,t)e^{-\fb U_{j,q}(\x_2,s,h)}\mathcal{Q}(\x_1,\x_2,s)\Phi_1(\x,s)\mu(\dd\x_2)\right)\right|\\
	\le C C_M(\x_1,s) e^{-2s}e^{\fb(j-1)(h-s)}e^{-2(j-1)h}.
\end{multline*}
Similarly applying Lemma~\ref{lem:energy-multiplier} to the integral in the brackets and applying the induction hypothesis on $i$, we have
\begin{multline}\label{child-time-bound}
	\left|\int_{\Lambda^i}F^{(i,p)}(\x_1,v,t)e^{-\fb U_{i,p}(\x_1,s,v)}\int_{\Lambda^j}F^{(j,q)}(\x_2,h,t)e^{-\fb U_{j,q}(\x_2,s,h)}\mathcal{Q}(\x_1,\x_2,s)\Phi_1(\x,s)\mu(\dd\x)\right|\\
	\le Ce^{-2s}e^{\fb(i-1)(v-s)}e^{-2(i-1)v}e^{\fb(j-1)(h-s)}e^{-2(j-1)h}\int_{\Lambda}f(x,s)\mu(\dd x).
\end{multline}
Integrating $v$ and $h$ over $[s,t]$ produces Eq.~\eqref{propz1} in the case $a=1$. In the case $a=2$, we instead start from the function $(\x_1,v)\mapsto D_N\left(e^{-\fb U_{i,p}(\x_1,s,v)}\mathcal{Q}(\x_1,\x_2,s)\Phi_2(\x,s)\right)$, apply the induction hypothesis on $i$ first, and then $j$ to obtain Eq.~\eqref{propz1}. Adding both cases of $a$ together proves Eq.~\eqref{propz1} in the case $i=2p$ and $j=2q$, provided Eq.~\eqref{jgoal-split} holds.

To that end, it suffices to prove
\begin{multline}\label{jgoal}
 \left|D_{M\cup N}
 [\mathcal Q(\x_1,\x_2,s)\Phi(\x,s)]\right|\le CF(\x,s)G_{i,j}(\x_1,\x_2,s)
 \chi_M(\x_1,s)\chi_N(\x_2,s)\\
 \times\prod_{(m_1,m_2)\in M}[1\wedge e^s|x_{m_1}-x_{m_2}|]^\gamma
 \prod_{(b_1,b_2)\in N}[1\wedge e^s|x_{b_1}-x_{b_2}|]^\gamma,
\end{multline}
since by multiplying both sides of Eq.~\eqref{jgoal} by $F_a/F$ we get Eq.~\eqref{jgoal-split}. Furthermore, like in Lemma~\ref{lem:energy-multiplier}, we can assume $|x_{m_1}-x_{m_2}|<e^{-s}$ for all $(m_1,m_2)\in M$ and similarly for $N$. In the following, we write $\sigma_{1}$ and $\sigma_{2}$ for the permutations associated to $m(\x_1)$ and $m(\x_2)$, respectively. Define $T_1(\sigma_1)=\prod_{k=1}^p\tau_{(k,p+\sigma_1(k))}$ and similarly for $T_2(\sigma_2)$. We have the identity
\begin{equation}\label{newphi}
\Phi(\x,s)=(1-T_2(\sigma_2))(1-T_1(\sigma_1))\left[\phi_{i+j}(T^+\x,s)\right]
\end{equation}
due to the symmetry of $\phi_{i+j}$ (item~\ref{test:symmetry} of Definition~\ref{TestDef}). When estimating a finite difference at $\x$, we choose matchings optimal at $\x$ and keep them fixed throughout that calculation. Thus, in the following we replace $T_1$ by $T_1(\sigma_1)$ and similarly for $T_2$. By telescoping
\begin{equation}\label{charge-reversal-telescoping}
1-T_1=\sum_{k=1}^p\tau_{(1,p+\sigma_{1}(1))}...\tau_{(k-1,p+\sigma_1(k-1))}(1-\tau_{(k,p+\sigma_1(k))}).
\end{equation}
Using the analogous formula for $T_2$, we apply the property~\eqref{testprop} to $\phi_{i+j}$ in Eq.~\eqref{newphi} to obtain
\begin{equation}\label{Phi-zero}
 \left|\Phi(\x,s)\right|\le CF(\x,s)[1\wedge m(\x_1)e^s]^\gamma[1\wedge m(\x_2)e^s]^\gamma.
\end{equation}
Combining Eqs.~\eqref{Q-two-neutral} and~\eqref{Phi-zero}, we have
\begin{equation*}
\left|\mathcal{Q}(\x_1,\x_2,s)\Phi(\x,s)\right|
\le CF(\x,s)
[1\wedge m(\x_1)e^s]^{1+\gamma}[1\wedge m(\x_2)e^s]^{1+\gamma}G_{i,j}(\x_1,\x_2,s).
\end{equation*}
Since $1+\gamma\ge2\gamma$, this proves Eq.~\eqref{jgoal} for $M=N=\emptyset$. Next, write $T_1a=(T_1(a_1),T_1(a_2))$, with the endpoints put in increasing order, and $T_1A=\{T_1a|a\in A\}$. For $a\in M$, since $T_1^2=\mathrm{Id}$ and $T_1\tau_aT_1=\tau_{T_1a}$, direct expansion gives
\begin{equation*}
 D_a(1-T_1)=D_a-T_1D_{T_1a},
\end{equation*}
and similarly for swaps in $N$ and $T_2$. Moving $T_1$  successively to the left using $D_aT_1=T_1 D_{T_1 a}$ (and similarly for $T_2$) gives from Eq.~\eqref{newphi}, for a nonempty set $A\subseteq M\cup N$,
\begin{equation}\label{Phi-conjugated-expansion}
D_A\Phi(\x,s)
=\left[D_{A\cap M}-T_1D_{T_1(A\cap M)}\right]
\left[D_{A\cap N}-T_2D_{T_2(A\cap N)}\right]
\left[\phi_{i+j}(T^+\x,s)\right].
\end{equation}
This expansion also holds if $A\cap M$ or $A\cap N$ is empty, with the empty product interpreted as the identity. The property~\eqref{testprop}, together with Eq.~\eqref{charge-reversal-telescoping}, applied to the terms of Eq.~\eqref{Phi-conjugated-expansion} gives
\begin{equation}\label{Phi-differences}
 \left|D_A\Phi(\x,s)\right|
\le CF(\x,s)[1\wedge m(\x_1)e^s]^{\gamma(1-|A\cap M|)_+}[1\wedge m(\x_2)e^s]^{\gamma(1-|A\cap N|)_+}
\prod_{(a_1,a_2)\in A}\bigl(|x_{a_1}-x_{a_2}|e^s\bigr)^\gamma.
\end{equation}
This also holds for $A=\emptyset$, as established in Eq.~\eqref{Phi-zero}. Recalling the derivation of Eq.~\eqref{Q-two-neutral}, we get similarly for $m\in M$
\begin{multline}\label{Q-swap}
|D_m\mathcal{Q}(\x_1,\x_2,s)|
=\left|(\la_{m_1}^{(p)}-\la_{m_2}^{(p)})\sum_{r=i+1}^{i+j}\la_{r-i}^{(q)}
\bigl(Q_s(x_r,x_{m_1})-Q_s(x_r,x_{m_2})\bigr)\right|\\
\le C[1\wedge m(\x_2)e^s]|x_{m_1}-x_{m_2}|e^sG_{i,j}(\x_1,\x_2,s).
\end{multline}
The same holds for $m\in N$ replacing $m(\x_2)$ with $m(\x_1)$. For $a\in M$ and $b\in N$, we have
\begin{multline}\label{Q-mixed-swap}
|D_{\{a,b\}}\mathcal{Q}(\x_1,\x_2,s)|
\le4\left|Q_s(x_{a_1},x_{b_1})-Q_s(x_{a_2},x_{b_1})-Q_s(x_{a_1},x_{b_2})+Q_s(x_{a_2},x_{b_2})\right|\\
\le C|x_{b_1}-x_{b_2}|e^s|x_{a_1}-x_{a_2}|e^sG_{i,j}(\x_1,\x_2,s).
\end{multline}
For any third $c\in M$ or $c\in N$ disjoint from $a,b$ we have $D_{\{a,b,c\}}\mathcal{Q}(\x_1,\x_2,s)=0$. Thus if $|A\cap N|>1$ or $|A\cap M|>1$
\begin{equation}\label{Q-higher-differences}
D_A\mathcal{Q}(\x_1,\x_2,s)=0	.
\end{equation}
Then using Eq.~\eqref{app:product-rule-first} gives
\begin{equation}\label{Q-Phi-Leibniz}
D_{M\cup N}
[\mathcal{Q}(\x_1,\x_2,s)\Phi(\x,s)]
=\sum_{A\subseteq M\cup N}
\left[D_A\mathcal{Q}(\x_1,\x_2,s)\right]
\left[\tau_AD_{(M\cup N)\setminus A}\Phi(\x,s)\right].
\end{equation}
By Eq.~\eqref{Q-higher-differences}, the right-hand side of Eq.~\eqref{Q-Phi-Leibniz} reduces to
\begin{multline}\label{Q-Phi-expansion}
\mathcal{Q}(\x_1,\x_2,s)
\left[D_{M\cup N}\Phi(\x,s)\right]
+\sum_{a\in M\cup N}[D_a\mathcal{Q}(\x_1,\x_2,s)]
\left[\tau_aD_{(M\cup N)\setminus\{a\}}\Phi(\x,s)\right]\\
+\sum_{\substack{a\in M\\ b\in N}}[D_{\{a,b\}}\mathcal{Q}(\x_1,\x_2,s)]
\left[\tau_{\{a,b\}}D_{(M\cup N)\setminus\{a,b\}}\Phi(\x,s)\right].
\end{multline}
The first term is bounded using Eqs.~\eqref{Q-two-neutral} and~\eqref{Phi-differences}:
\begin{multline}\label{first-term-bound}
\left|\mathcal{Q}(\x_1,\x_2,s)D_{M\cup N}\Phi(\x,s)\right|\le CG_{i,j}(\x_1,\x_2,s)F(\x,s)\\
\times[1\wedge m(\x_1)e^s]^{1+\gamma(1-|M|)_+}[1\wedge m(\x_2)e^s]^{1+\gamma(1-|N|)_+}
\prod_{(a_1,a_2)\in M\cup N}\bigl(|x_{a_1}-x_{a_2}|e^s\bigr)^\gamma.
\end{multline}
Since $[1\wedge m(\x_1)e^s]^{1+\gamma(1-|M|)_+}\le\chi_M(\x_1,s)$, and likewise for $N$ and $\x_2$, Eq.~\eqref{first-term-bound} gives the bound in Eq.~\eqref{jgoal}, using Eq.~\eqref{chi-definition}. For the second term in Eq.~\eqref{Q-Phi-expansion}, consider $a\in N$, and use Eqs.~\eqref{Q-swap} and~\eqref{Phi-differences}:
\begin{multline}\label{second-term-bound}
\left|[D_a\mathcal{Q}(\x_1,\x_2,s)]
\left[\tau_aD_{(M\cup N)\setminus\{a\}}\Phi(\x,s)\right]\right|\le CG_{i,j}(\x_1,\x_2,s)F(\x,s)\\
\times[1\wedge m(\x_1)e^s]^{1+\gamma(1-|M|)_+}[1\wedge m(\tau_a\x_2)e^s]^{\gamma(2-|N|)_+}
\prod_{(b_1,b_2)\in M\cup N}\bigl(|x_{b_1}-x_{b_2}|e^s\bigr)^\gamma.
\end{multline}
In the region $m(\tau_a\x_2)<m(\x_2)$ we have $[1\wedge m(\tau_a\x_2)e^s]^{\gamma(2-|N|)_+}\le[1\wedge m(\x_2)e^s]^{\gamma(2-|N|)_+}\le \chi_{N}(\x_2,s)$.
In the complementary region $m(\tau_a\x_2)\ge m(\x_2)$, Eq.~\eqref{matching-Lipschitz} and its symmetric version give
$|m(\tau_a\x_2)-m(\x_2)|\le2|x_{a_1}-x_{a_2}|$. Hence
\begin{equation}\label{permuted-chi-bound}
\begin{split}
[1\wedge m(\tau_a\x_2)e^s]^{\gamma(2-|N|)_+}
&\le\bigl([1\wedge m(\x_2)e^s]+2[1\wedge|x_{a_1}-x_{a_2}|e^s]\bigr)^{\gamma(2-|N|)_+}\\
&\le2\chi_N(\x_2,s),
\end{split}
\end{equation}
since $|N|\ge1$. Thus Eq.~\eqref{second-term-bound} satisfies Eq.~\eqref{jgoal}. The case $a\in M$ is treated in the same way.

For the third term in Eq.~\eqref{Q-Phi-expansion}, Eqs.~\eqref{Q-mixed-swap} and~\eqref{Phi-differences} give, for $a\in M$ and $b\in N$,
\begin{multline}\label{third-term-bound}
\left|[D_{\{a,b\}}\mathcal{Q}(\x_1,\x_2,s)]
\left[\tau_{\{a,b\}}D_{(M\cup N)\setminus\{a,b\}}\Phi(\x,s)\right]\right|\le CG_{i,j}(\x_1,\x_2,s)F(\x,s)\\
\times[1\wedge m(\tau_a\x_1)e^s]^{\gamma(2-|M|)_+}[1\wedge m(\tau_b\x_2)e^s]^{\gamma(2-|N|)_+}
\prod_{(m_1,m_2)\in M\cup N}\bigl(|x_{m_1}-x_{m_2}|e^s\bigr)^\gamma.
\end{multline}
Applying the comparison~\eqref{permuted-chi-bound} separately for $\x_1$ and $\x_2$ shows that Eq.~\eqref{third-term-bound} also satisfies the bound in Eq.~\eqref{jgoal}, proving it wholly.

For cases other than $(i,j)=(2p,2q)$ we use a similar strategy. If $i=2p$ and $j\ne2q$, set
\[
\Phi(\x,s)=(1-T_1)\phi_{i+j}(T^+\x,s),
\]
 use
Eq.~\eqref{Q-one-neutral}, and omit the factor $\chi_N$.
Keep the decomposition $\Phi=\Phi_1+\Phi_2$ and the same
integration orders, applying the second item of
Lemma~\ref{lem:energy-multiplier} in $\x_1$ and its first item
with $\kappa_{j,q}=1$ in $\x_2$.
The case $i\ne2p$ and $j=2q$ is symmetrical to the previous case.
If $i\ne2p$ and $j\ne2q$, estimate the two contributions in
Eq.~\eqref{ijodd} separately, with
$\Phi(\x,s)=\phi_{i+j}(T^\pm\x,s)$; use Eq.~\eqref{Q-basic},
omit both $\chi$-factors, and apply the first item of the lemma
with $\kappa_{i,p}=\kappa_{j,q}=1$.
For $i,j\ge2$, the resulting time bound is again
Eq.~\eqref{child-time-bound}.

If $i=1$, use $\mathcal F^{(1,0)}=1$ and integrate directly
in $x_1$, with no integration in $v$; the case $j=1$ is analogous.
Keep the integration order prescribed by $\Phi_a$.
The first spatial integration supplies $e^{-2s}$, and,
when $i+j>2$, the remaining time integral gives the denominator
$(i+j-2)(2-\fb)$.
Summing the finitely many contributions proves
Eq.~\eqref{propz1}.

For Eq.~\eqref{pairing-limit}, expand the recursive time
integrals, keeping the sums in
Eqs.~\eqref{ieven}--\eqref{ijeven} together.
We repeat the preceding induction with the spatial integration, now with all recursive time
variables held fixed, and not integrated over until the very end. Furthermore, we write each time variable as the time at the preceding recursive step (the lower bound of the corresponding integration domain) plus an increment $r\ge0$. The induction yields the right-hand side of
Eq.~\eqref{child-time-bound} multiplied by a product of
factors $e^{-(2-\fb)(k_i-1)r_i}$, each $r_i$ corresponding to a lower-order term,
and $k_i\ge2$. Together with the exponential factors
already present in Eq.~\eqref{child-time-bound}, these
give a bound independent of $t$ and integrable over each
$r_i$, since $2-\fb>0$. The only $t$ dependence for this bound is then contained in the upper limits for each $r_i$. Dominated convergence therefore proves
Eq.~\eqref{pairing-limit}.
\end{proof}

\section{Gaussian estimates and test functions}
We now verify that the random functions appearing in the martingale representation satisfy the test-function estimates used above almost surely. We first establish a spatial Hölder bound for the Gaussian field and estimate the moments of its coefficient under the tilted measures, for fixed $t$ and $\alpha\in\R$,\begin{equation*}
 \frac{\dd\P_{\alpha,t}}{\dd\P}\coloneqq\frac{e^{\alpha M_t}}{\E[e^{\alpha M_t}]},
 \qquad \E_{\alpha,t}[\cdot]\coloneqq\frac{\E[\cdot\,e^{\alpha M_t}]}{\E[e^{\alpha M_t}]}.
\end{equation*}

\begin{lemma}\label{lem:Holder-moments}
For $0<\gamma<1$, the field $(X_t(x))_{x\in \Lambda,t\ge0}$ has, for every $u\ge0$, an $\mathcal F_u$-measurable coefficient
$H_u^{(\gamma)}\ge1$, independent
of $\alpha$, such that
\begin{equation}\label{Holder-increments}
 |X_u(x)-X_u(y)|\le H_u^{(\gamma)}(e^u|x-y|)^\gamma,
 \qquad x,y\in \Lambda.
\end{equation}
For every $u,t\ge0$, $\alpha\in\R$ and real $n>0$,
\begin{equation}\label{Holder-moment-unnormalized}
 \E[e^{\alpha M_t}(H_u^{(\gamma)})^n]
 \le C_{\gamma,n}e^{n(1-\gamma)u}\E[e^{|\alpha|M_t}].
\end{equation}
For $\alpha\ge0$, we also have
\begin{equation}\label{Holder-moment-n}
\E_{\alpha,t}[(H_u^{(\gamma)})^n]
 \le C_{\gamma,n}e^{n(1-\gamma)
 u}.
\end{equation}
\end{lemma}
\begin{proof}
The moments of $M_t$ for integers $n\ge1$ are
\begin{equation}\label{moment-expansion}
 \E[M_t^n]=\frac{1}{2^n}\sum_{\la\in\{-1,1\}^n}\int_{\Lambda^n}
 \E\left[\cos\left(\gb\sum_{k=1}^n\la_kX_t(x_k)\right)\right]
 e^{\fb\sum_{k=1}^nK_t(x_k,x_k)}\mu(\dd\x).
\end{equation}
For centred jointly Gaussian $X,Y$,
\begin{equation*}
 \E[e^{\la Y+\mathrm i\gb X}]
 =e^{\frac{\la^2}{2}\E[Y^2]-\fb\E[X^2]}
   e^{\mathrm i\la\gb\E[XY]}.
\end{equation*}
Taking real parts and absolute values, and using $\E[\cos(\gb X)]=e^{-\fb\E[X^2]}$, gives
\[
 |\E[\cos(\gb X)e^{\la Y}]|\le e^{\la^2\E[Y^2]/2}\E[\cos(\gb X)].
\]
For every centred Gaussian linear functional $Y$ of the field, Eq.~\eqref{moment-expansion} therefore yields
\begin{equation*}
 |\E[e^{\la Y}M_t^n]|\le e^{\frac{\la^2}{2}\E[Y^2]}\E[M_t^n].
\end{equation*}
Expanding $e^{\alpha M_t}$ and applying this bound term by term gives, for every $\alpha\in\R$,
\begin{equation}\label{Gaussian-domination}
 \E[e^{\alpha M_t+\la Y}]
 \le e^{\frac{\la^2}{2}\E[Y^2]}\E[e^{|\alpha|M_t}].
\end{equation}
Now taking $Y=X_u(x)-X_u(y)$, we obtain from Eq.~\eqref{laderive} and the mean value theorem
\begin{align}\label{increment-variance}
 \E[(X_u(x)-X_u(y))^2]
 &=\int_0^u\bigl(Q_v(x,x)+Q_v(y,y)-2Q_v(x,y)\bigr)\dd v\le C (e^{u}|x-y|)^{2}.
\end{align}
Consequently, for $r>0$, applying Eq.~\eqref{Gaussian-domination} gives
\begin{equation*}
 \E[e^{\alpha M_t}\1_{\{|X_u(x)-X_u(y)|\ge r\}}]
 \le2\exp\left(-\frac{r^2}{2C(e^{u}|x-y|)^{2}}\right)\E[e^{|\alpha|M_t}].
\end{equation*}
From the tail we have, for every $n\ge1$,
\begin{equation*}
 \E[e^{\alpha M_t}|X_u(x)-X_u(y)|^n]
 \le C_n(e^u|x-y|)^n\E[e^{|\alpha|M_t}].
\end{equation*}

The Kolmogorov--Totoki argument
\cite[Section~1.8, Lemmas~1.8.1--1.8.2]{Kunita2019}, applied with the preceding moment bound, gives, for $0<\gamma<1$, an $\mathcal F_u$-measurable coefficient $H_u^{(\gamma)}\ge1$ satisfying Eq.~\eqref{Holder-increments} and, for $n\ge 3\vee \frac{2}{1-\gamma}$, Eq.~\eqref{Holder-moment-unnormalized}.
H\"older's inequality extends the bound to all real $n>0$.
For $\alpha\ge0$, dividing by $\E[e^{\alpha M_t}]$ gives Eq.~\eqref{Holder-moment-n}.
\end{proof}
To apply Proposition~\ref{prop:pairing} to the random functions in the martingale representation~\eqref{martingale-representation}, we verify that the following functions belong to the class of test functions.
\begin{lemma}\label{lem:cosine-test}
Fix $\gb^2/4<\gamma<1$, $i\ge2$, and
$0\le p\le\lfloor i/2\rfloor$.
For every sample path satisfying Eq.~\eqref{Holder-increments},
the following functions belong to
$\mathcal D_{i,p}^\gamma(\Lambda^i\times\R_+)$:
\begin{enumerate}
\item
\begin{equation*}
 e^{\fb\sum_{a=1}^iK_u(x_a,x_a)}
 \cos(\gb Y_p(\x,u))
\end{equation*}
with majorant
$Ce^{\fb iu}(H_u^{(\gamma)})^{\lfloor i/2\rfloor}$.

\item
\begin{equation*}
 e^{\fb\sum_{a=1}^iK_u(x_a,x_a)}
 \left[
 \cos(\gb Y_p(\x,u))
 -e^{-\fb\E[Y_p(\x,u)^2]}
 \right]
\end{equation*}
with majorant $Ce^{\fb iu}(H_u^{(\gamma)})^i$.
If $i=2p$, it satisfies Eq.~\eqref{unweighted-property}
with this majorant.

\item For $z\in \Lambda$,
\begin{equation*}
 e^{\fb\sum_{a=1}^iK_u(x_a,x_a)}
 \sin(\gb Y_p(\x,u))
 \left(\sum_{a=1}^i\la_a^{(p)}Q_u(z,x_a)\right)
\end{equation*}
with majorant
\[
 Ce^{\fb iu}(H_u^{(\gamma)})^i
 \sum_{a=1}^i(1+e^u|z-x_a|)^{-3}.
\]
If $i=2p$, it satisfies Eq.~\eqref{unweighted-property}
with this majorant.
The same assertions hold when $Q_u(z,x_a)$ is replaced by
$e^{-u}\partial_{z_j}Q_u(z,x_a)$, for $j=1,2$.
\end{enumerate}
\end{lemma}

\begin{proof}
All the displayed tests satisfy the symmetry condition in
item~\ref{test:symmetry} of Definition~\ref{TestDef}.
The factor $e^{\fb\sum_{a=1}^iK_u(x_a,x_a)}$ is invariant
under all coordinate permutations and is at most
$Ce^{\fb iu}$ by Eq.~\eqref{log-covariance-bound}.
We therefore omit this factor from the estimates below
and restore it at the end.

Starting from the first item, we have
$|\cos(\gb Y_p(\x,u))|\le1\le
(H_u^{(\gamma)})^{\lfloor i/2\rfloor}$.
This proves the case $M=\emptyset$.
For $|M|\ge1$, take $(a,b)\in M$.
If $\la_a^{(p)}=\la_b^{(p)}$, the corresponding difference
is zero. If $\la_a^{(p)}=-\la_b^{(p)}$, then
\begin{equation}\label{phase-one-swap}
 \left|D_{(a,b)}
 e^{\mathrm i\gb(\la_a^{(p)}X_u(x_a)
                    +\la_b^{(p)}X_u(x_b))}\right|
 =2|\sin(\gb(X_u(x_a)-X_u(x_b)))|
 \le4(1\wedge|X_u(x_a)-X_u(x_b)|).
\end{equation}
Now factor $e^{\mathrm i\gb Y_p(\x,u)}$ into exponentials
of the above form for each pair in $M$, and the factor
from the remaining coordinates. Each difference affects
only its own factor, so Eq.~\eqref{phase-one-swap} gives
\begin{equation}\label{phase-mixed-swaps}
 \left|D_Me^{\mathrm i\gb Y_p(\x,u)}\right|
 \le4^{|M|}
 \prod_{(m_1,m_2)\in M}
 (1\wedge|X_u(x_{m_1})-X_u(x_{m_2})|).
\end{equation}
Taking real parts, applying Eq.~\eqref{Holder-increments},
and using $H_u^{(\gamma)}\ge1$ and
$|M|\le\lfloor i/2\rfloor$, we obtain
\begin{equation}\label{cosine-test-bound}
 \left|D_M\cos(\gb Y_p(\x,u))\right|
 \le C(H_u^{(\gamma)})^{\lfloor i/2\rfloor}
 \prod_{(m_1,m_2)\in M}
 (1\wedge e^u|x_{m_1}-x_{m_2}|)^\gamma.
\end{equation}
This proves the first item. Taking imaginary parts gives
the same estimate for the sine.

For the second item, first suppose $i=2p$.
Let $\sigma_\x$ attain the minimum in
Eq.~\eqref{matching-distance}. Then
\begin{equation}\label{phase-matching}
 Y_p(\x,u)
 =\sum_{k=1}^p
 \left(X_u(x_{p+\sigma_\x(k)})-X_u(x_k)\right).
\end{equation}
We write
\[
 \cos(\gb Y_p(\x,u))-e^{-\fb\E[Y_p(\x,u)^2]}
 =[\cos(\gb Y_p(\x,u))-1]
  +[1-e^{-\fb\E[Y_p(\x,u)^2]}].
\]
For the first term, Eqs.~\eqref{phase-matching}
and~\eqref{Holder-increments}, together with
$|x_k-x_{p+\sigma_\x(k)}|\le m(\x)$, give
\begin{equation}\label{phase-matching-bound}
 1\wedge|Y_p(\x,u)|
 \le CH_u^{(\gamma)}[1\wedge e^um(\x)]^\gamma.
\end{equation}
Consequently, by trigonometry
\begin{equation}\label{cosine-zero-bound}
 |\cos(\gb Y_p(\x,u))-1|
 \le2(1\wedge|\gb Y_p(\x,u)|)^2
 \le C(H_u^{(\gamma)})^2
 [1\wedge e^um(\x)]^{2\gamma}.
\end{equation}
For the second term, the triangle inequality and
Eq.~\eqref{increment-variance} give
\begin{equation}\label{phase-matching-variance}
 \E[Y_p(\x,u)^2]^{1/2}
 \le\sum_{k=1}^p
 \E[|X_u(x_{p+\sigma_\x(k)})-X_u(x_k)|^2]^{1/2}
 \le Ce^um(\x).
\end{equation}
Thus
\[
 1-e^{-\fb\E[Y_p(\x,u)^2]}
 \le1\wedge\bigl(\fb\E[Y_p(\x,u)^2]\bigr)
 \le C[1\wedge e^um(\x)]^2
 \le C[1\wedge e^um(\x)]^{2\gamma}.
\]
Combining this with Eq.~\eqref{cosine-zero-bound} proves
the required estimate for $M=\emptyset$.

For $|M|\ge1$, use
\begin{equation}\label{cosine-expectation-identity}
 \cos(\gb Y_p(\x,u))-e^{-\fb\E[Y_p(\x,u)^2]}
 =\cos(\gb Y_p(\x,u))-\E[\cos(\gb Y_p(\x,u))].
\end{equation}
If $M=\{a\}$, trigonometry gives
\begin{equation}\label{cosine-one-swap}
 |D_a\cos(\gb Y_p(\x,u))|
 =2\left|\sin\left(\frac{\gb}{2}
       (Y_p(\x,u)-Y_p(\tau_a\x,u))\right)\right|\left|\sin\left(\frac{\gb}{2}
       (Y_p(\x,u)+Y_p(\tau_a\x,u))\right)\right|.
\end{equation}
Since
\[
 |Y_p(\x,u)-Y_p(\tau_a\x,u)|
 \le2|X_u(x_{a_1})-X_u(x_{a_2})|,
\]
Eq.~\eqref{Holder-increments} gives
\begin{equation}\label{phase-difference-sine}
 \left|\sin\left(\frac{\gb}{2}
       (Y_p(\x,u)-Y_p(\tau_a\x,u))\right)\right|
 \le C H_u^{(\gamma)}
 [1\wedge e^u|x_{a_1}-x_{a_2}|]^\gamma.
\end{equation}
For the second sine, apply Eq.~\eqref{phase-matching-bound}
at $\x$ and $\tau_a\x$. By Eq.~\eqref{matching-Lipschitz},
\begin{equation}\label{matching-one-swap}
 1\wedge e^um(\tau_a\x)
 \le[1\wedge e^um(\x)]
       +2[1\wedge e^u|x_{a_1}-x_{a_2}|].
\end{equation}
Therefore
\begin{equation}\label{phase-sum-sine}
 \left|\sin\left(\frac{\gb}{2}
       (Y_p(\x,u)+Y_p(\tau_a\x,u))\right)\right|
 \le CH_u^{(\gamma)}
 \left([1\wedge e^um(\x)]
       +[1\wedge e^u|x_{a_1}-x_{a_2}|]\right)^\gamma.
\end{equation}
Combining Eqs.~\eqref{cosine-one-swap},
\eqref{phase-difference-sine} and~\eqref{phase-sum-sine},
we obtain
\begin{equation}\label{cosine-one-swap-bound}
 |D_a\cos(\gb Y_p(\x,u))|
 \le C(H_u^{(\gamma)})^2
 \left([1\wedge e^um(\x)]
       +[1\wedge e^u|x_{a_1}-x_{a_2}|]\right)^\gamma[1\wedge e^u|x_{a_1}-x_{a_2}|]^\gamma.
\end{equation}

For the expectation, Eqs.~\eqref{phase-matching-variance}
and~\eqref{increment-variance} give
\begin{align*}
 \E[|Y_p(\x,u)-Y_p(\tau_a\x,u)|^2]^{1/2}
 &\le Ce^u|x_{a_1}-x_{a_2}|,\\
 \E[|Y_p(\x,u)+Y_p(\tau_a\x,u)|^2]^{1/2}
 &\le Ce^u\bigl(m(\x)+|x_{a_1}-x_{a_2}|\bigr).
\end{align*}
Applying Cauchy--Schwarz in Eq.~\eqref{cosine-one-swap}
and using $|\sin v|\le1\wedge|v|$, we obtain
\begin{equation*}
 |D_a\E[\cos(\gb Y_p(\x,u))]|
 \le C
 \left([1\wedge e^um(\x)]
       +[1\wedge e^u|x_{a_1}-x_{a_2}|]\right)^\gamma[1\wedge e^u|x_{a_1}-x_{a_2}|]^\gamma.
\end{equation*}
Together with Eqs.~\eqref{cosine-expectation-identity}
and~\eqref{cosine-one-swap-bound}, this proves the required
estimate for $M=\{a\}$.

For $|M|\ge2$, Eq.~\eqref{increment-variance} and Gaussian
moments give, for every integer $r\ge1$,
\begin{equation*}
 \E[(1\wedge|X_u(x)-X_u(y)|)^r]
 \le C(1\wedge e^u|x-y|)^{r\gamma}.
\end{equation*}
Since the swaps commute with expectation, H\"older's
inequality in Eq.~\eqref{phase-mixed-swaps} gives
\[
 |D_M\E[\cos(\gb Y_p(\x,u))]|
 \le C\prod_{(m_1,m_2)\in M}
 (1\wedge e^u|x_{m_1}-x_{m_2}|)^\gamma.
\]
Combining this with Eq.~\eqref{cosine-test-bound} proves
Eq.~\eqref{unweighted-property}, since $\chi_M=1$ for
$|M|\ge2$. In all three cases, the coefficient is at most
$C(H_u^{(\gamma)})^i$. If $i\ne2p$, the preceding argument for $|M|\ge2$ applies similarly, but also for the case $|M|=1$. For $M=\emptyset$ we use
$|\cos(\gb Y_p)-\E[\cos(\gb Y_p)]|\le2$.
This proves the second item.

For the third item, Eq.~\eqref{laderive} and symmetry give
\begin{equation*}
 |Q_u(z,x)|+e^{-u}|\nabla_xQ_u(z,x)|
 +e^{-u}|\partial_{z_j}Q_u(z,x)|\\
 +e^{-2u}|\nabla_x\partial_{z_j}Q_u(z,x)|
 \le C(1+e^u|z-x|)^{-3}.
\end{equation*}
Thus $Q_u(z,\cdot),$
$e^{-u}\partial_{z_j}Q_u(z,\cdot)$ and their scaled gradients satisfy the same bound. Proving the assertion for $Q_u$ then proves the claim for $e^{-u}\partial_{z_j}Q_u(z,\cdot)$.
Applying the mean value theorem in the case $e^u|x_{a_1}-x_{a_2}|\le1$ gives
\begin{align*}
 \left|\sum_{b=1}^i\la_b^{(p)}Q_u(z,x_b)\right|
 &\le C\sum_{b=1}^i(1+e^u|z-x_b|)^{-3},\\
 \left|D_a\sum_{b=1}^i\la_b^{(p)}Q_u(z,x_b)\right|
 &\le C[1\wedge e^u|x_{a_1}-x_{a_2}|]
       \sum_{b=1}^i(1+e^u|z-x_b|)^{-3}.
\end{align*}
If $|M|\ge2$, $D_M\sum_{b=1}^i\la_b^{(p)}Q_u(z,x_b)=0$ since each
summand depends on only one coordinate.
The formula~\eqref{app:product-rule-first}, applied to $\sin(\gb Y_p(\x,u))\left(\sum_{b=1}^i\la_b^{(p)}Q_u(z,x_b)\right)$, therefore
has only terms with zero or one difference on the sum.
Taking imaginary parts in Eq.~\eqref{phase-mixed-swaps}
proves the
claim in the case $i\neq 2p$.

If $i=2p$, pairing the terms according to $\sigma_\x$ gives
\[
 \left|\sum_{b=1}^i\la_b^{(p)}Q_u(z,x_b)\right|
 \le C[1\wedge e^um(\x)]
       \sum_{b=1}^i(1+e^u|z-x_b|)^{-3}.
\]
For $M=\emptyset$, this bound and
Eq.~\eqref{phase-matching-bound}, paired with the inequality $|\sin w|\le1\wedge |w|$, give the factor
$[1\wedge e^um(\x)]^{1+\gamma}$, which is at most
$[1\wedge e^um(\x)]^{2\gamma}$.
For $M=\{a\}$, apply the formula
\eqref{app:one-swap-product-rule}, putting the sum first.
The term with the difference on the sine retains the
factor $[1\wedge e^um(\x)]$ from the sum.
The other term is bounded using
Eq.~\eqref{phase-matching-bound} at $\tau_a\x$ and
Eq.~\eqref{matching-one-swap}.
Both terms are at most
\[
 C(H_u^{(\gamma)})^i
 \left(\sum_{b=1}^i(1+e^u|z-x_b|)^{-3}\right)
 \chi_{\{a\}}(\x,u)
 [1\wedge e^u|x_{a_1}-x_{a_2}|]^\gamma.
\]
For $|M|\ge2$, the same estimate as for $i\neq2p$ suffices, since
$\chi_M=1$.
\end{proof}

\section{Proof of Theorem \ref{maintheorem}}
Fix $\alpha\in\R$, $\gb^2<4$ and an integer $n\ge1$ such that
\[
 4-\frac{2}{n-1}\le\gb^2<4-\frac2n,
\]
(for $n=1$ we assume $0\le\gb^2<2$).
Furthermore, choose $\gb^2/4<\gamma<1$ and $\eta$ so that
\[
 1-\gamma<\eta<4-\frac2n-\gb^2.
\]
Then $\eta<2-\fb$.
Fix $s_0<\infty$ and let $F$ be a bounded real $\mathcal F_{s_0}$-measurable random variable, and take $t\ge s\ge s_0$. From Eq.~\eqref{partition-expansion}, recall the formula
\[
\begin{aligned}
 Z_{\alpha,\beta,t}(F)=\E\left[Fe^{N_t^{(2n-1,\alpha,t)}-\langle N^{(2n-1,\alpha,t)}\rangle_t/2+Q_{t,t}^{(2n-1,\alpha)}}\right].
\end{aligned}
\]
Notice that we can assume $F\ge1$ without loss of generality, since we can always add a sufficiently large constant $c$ to real-valued $F$: $Z_{\alpha,\beta,t}(F)=Z_{\alpha,\beta,t}(F+c)-cZ_{\alpha,\beta,t}(1)$. Complex-valued $F$ can then be treated by considering the real and imaginary parts separately.

 Now define
\begin{equation*}
 V_F(s,t)\coloneqq
 \ln\E\left[
 F\1_{E_s}e^{N_s^{(2n-1,\alpha,t)}-\langle N^{(2n-1,\alpha,t)}\rangle_s/2+Q_{s,t}^{(2n-1,\alpha)}}
 \right],
 \qquad
 E_s\coloneqq\{H_s^{(\gamma)}<e^{\eta s}\},
\end{equation*}
where $H_s^{(\gamma)}$ is from Eq.~\eqref{Holder-increments}. We aim to show

\begin{equation}\label{estimateA}
\ln Z_{\alpha,\beta,t}(F)=V_F(s,t)+o_s(1)
\end{equation}
 uniformly in $t\ge s$. Theorem \ref{maintheorem} then follows from the next lemma and an extension argument to bounded \(\mathcal F_\infty\)-measurable functions.

\begin{lemma}\label{lem:fixed-scale-convergence}
For every sufficiently large fixed $s\ge s_0$, $V_F(s,t)$ converges to a finite limit as $t\to\infty$.
\end{lemma}
\begin{proof}
Fix $s$ sufficiently large so that $\P(E_s)>0$.
Using Eqs.~\eqref{remainder-definition}, \eqref{A-representation}, \eqref{A-bracket} and~\eqref{martingale-representation}, together with the definitions of $M^{(i,t)}$, we have for the exponent in the definition of $V_F(s,t)$
\begin{equation}\label{fixed-scale-exponent}
\begin{aligned}
 &N_s^{(2n-1,\alpha,t)}-\frac12\langle N^{(2n-1,\alpha,t)}\rangle_s+Q_{s,t}^{(2n-1,\alpha)}
 =\alpha(M_s-\mu(\Lambda))\\
 &\quad+
 \sum_{i=2}^{2n-1}\alpha^i
 \sum_{p=0}^{\lfloor i/2\rfloor}
 \int_{\Lambda^i}\mathcal F^{(i,p)}(\x,s,t)
 e^{\fb\sum_{a=1}^iK_s(x_a,x_a)}
 \left[
 \cos(\gb Y_{p}(\x,s))
 -e^{-\fb\E[(Y_{p}(\x,s))^2]}
 \right]\mu(\dd\x)\\
 &\quad-
 \sum_{i=2}^{2n-1}\alpha^i
 \sum_{p=0}^{\lfloor i/2\rfloor}
 \int_0^s\int_{\Lambda^i}
 F^{(i,p)}(\x,u,t)e^{-\fb U_{i,p}(\x,u)}
 \mu(\dd\x)\dd u.
\end{aligned}
\end{equation}
The first term is independent of $t$. For the first integral, applying
Lemma~\ref{lem:cosine-test}, item~2, and
Lemma~\ref{lem:energy-multiplier}, we have the upper bound
\[
 \begin{aligned}
 &C e^{\fb i s}(H_s^{(\gamma)})^i
 \int_s^\infty
 e^{-2(i-1)u}e^{\fb(i-1)(u-s)}\dd u\le 
 C e^{[\fb i-2(i-1)]s}
 (H_s^{(\gamma)})^i
 <\infty.
 \end{aligned}
\]
Then dominated convergence and Eq.~\eqref{pairing-limit} imply the limit $t\to\infty$ exists for the term.
For the last integral in Eq.~\eqref{fixed-scale-exponent},
Lemma~\ref{lem:general-energy-multiplier}, applied to a constant test function, shows that
$e^{-\fb U_{i,p}(\x,u)}$ is a test with
majorant at most $Ce^{\fb i u}$.
For each fixed $u$, Eq.~\eqref{pairing-limit} gives convergence of the spatial integral. Proposition~\ref{prop:pairing} bounds its absolute value by $Ce^{[\fb i-2(i-1)]u}$, which is integrable over $[0,s]$. Dominated convergence therefore applies. Consequently, 
Eq.~\eqref{fixed-scale-exponent} converges almost surely. Moreover, the same estimates and the condition
$H_s^{(\gamma)}<e^{\eta s}$ on $E_s$ give
\begin{equation}\label{uniformestimate}
 \sup_{t\ge s}
 \left|N_s^{(2n-1,\alpha,t)}-\frac12\langle N^{(2n-1,\alpha,t)}\rangle_s+Q_{s,t}^{(2n-1,\alpha)}\right|
 \le C_s
 \qquad\text{on }E_s.
\end{equation}
Dominated convergence now shows that for sufficiently large fixed $s\ge0$
\[
 e^{V_F(s,t)}
 =
 \E\left[
 F\1_{E_s}
 e^{N_s^{(2n-1,\alpha,t)}-\langle N^{(2n-1,\alpha,t)}\rangle_s/2+Q_{s,t}^{(2n-1,\alpha)}}
 \right]
\]
converges in $t$. Its limit lies between
$e^{-C_s}\P(E_s)$ and $\|F\|_\infty e^{C_s}\P(E_s)$, so $V_F(s,t)$
converges to a finite limit for each sufficiently large $s$.
\end{proof}
We now move to proving Eq.~\eqref{estimateA} by bounding the ratio $e^{V_F(s,t)}/Z_{\alpha,\beta,t}(F)$ from above and below separately.

\begin{lemma}\label{lem:partition-lower-bound}

For every $\alpha\in\R$ and uniformly in $t\ge s$,
\begin{enumerate}
\item
\[
 \frac{e^{V_F(s,t)}}{Z_{\alpha,\beta,t}(F)}\le1+o_s(1)
\]
\item
\[
 \frac{e^{V_F(s,t)}}{Z_{\alpha,\beta,t}(F)}\ge1-o_s(1)
\]
\end{enumerate}
\end{lemma}
\begin{proof}[Proof of item 1]
By Eq.~\eqref{remainder-definition}, $|Q_{t,t}^{(2n-1,\alpha)}-Q_{s,t}^{(2n-1,\alpha)}|$ is bounded above by
a finite sum of terms of the following form, for
$2n\le k\le4n-2$, $k-2n+1\le l\le2n-1$ and
$0\le p\le\lfloor k/2\rfloor$:
\begin{equation}\label{remainder-integral-bound}
 \int_s^t\left|
 \int_{\Lambda^k}F^{(k,l,p)}(\x,u,t)
 e^{\fb\sum_{a=1}^kK_u(x_a,x_a)}
 \cos\left(\gb Y_{p}(\x,u)\right)\mu(\dd\x)
 \right|\dd u.
\end{equation}
 Define
\[
 S_{s,t}\coloneqq\{
 |X_u(x)-X_u(y)|
 \le3e^{\eta u}(e^u|x-y|)^\gamma,
 \ s\le u\le t,\ x,y\in \Lambda
 \}.
\]
On this event, tracing the proof of
Lemma~\ref{lem:cosine-test}, item~1, we obtain for the term
in Eq.~\eqref{remainder-integral-bound}
the upper bound
\begin{equation*}
 C\int_s^t e^{-2(k-1)u}e^{\fb ku}e^{\frac{k}{2}\eta u}\dd u
 \le C e^{-\left(2(k-1)-\frac{k}{2}(\gb^2+\eta)\right)s}.
\end{equation*}
Thus, on $S_{s,t}$,
\[
 |Q_{t,t}^{(2n-1,\alpha)}-Q_{s,t}^{(2n-1,\alpha)}|
 \le Ce^{-2n\left(2(1-\frac{1}{2n})
                    -\frac12(\gb^2+\eta)\right)s},
\]
and
\begin{align}
 Z_{\alpha,\beta,t}(F)
 &\ge
 \E\left[
 F\1_{E_s\cap S_{s,t}}
 e^{N_t^{(2n-1,\alpha,t)}-\frac12\langle N^{(2n-1,\alpha,t)}\rangle_t+Q_{t,t}^{(2n-1,\alpha)}}
 \right]\notag\\
 &\ge
 \E\left[
 F\1_{E_s\cap S_{s,t}}
 e^{N_t^{(2n-1,\alpha,t)}-\frac12\langle N^{(2n-1,\alpha,t)}\rangle_t+Q_{s,t}^{(2n-1,\alpha)}}
 \right]
 \exp\left(
 -Ce^{-2n\left(2(1-\frac{1}{2n})
                  -\frac12(\gb^2+\eta)\right)s}
 \right).
 \label{upper-Z-restriction}
\end{align}

The rest of the proof concerns estimating the remaining expectation. To this end, define for fixed $t\ge0$
\[
 \frac{\dd\Q}{\dd\P}
 =e^{N_t^{(2n-1,\alpha,t)}-\langle N^{(2n-1,\alpha,t)}\rangle_t/2},
 \qquad
 R_u(x)\coloneqq X_u(x)-\int_0^u\dd\langle X(x),N^{(2n-1,\alpha,t)}\rangle_v,
 \quad 0\le u\le t.
\]
By Girsanov's theorem, $R_u(x)$ is a continuous local
martingale under $\Q$, and
$\langle R(x),R(y)\rangle_u=K_u(x,y)$.
That is, $R$ is centred Gaussian under $\Q$, with covariance
$K_{u\wedge v}(x,y)$. Let
\[
 S^R_{s,t}\coloneqq\{
 |R_u(x)-R_u(y)-R_s(x)+R_s(y)|
 \le e^{\eta u}(e^u|x-y|)^\gamma,
 \ s\le u\le t,\ x,y\in \Lambda
 \}.
\]
We aim to show that
\begin{align}
\Q(S^R_{s,t})&\ge1-Ce^{-(\eta+\gamma-1)s},\label{gaussian-event-probability}\\
E_s\cap S^R_{s,t}&\subset S_{s,t}\label{drift-event-inclusion}.
\end{align}
It follows from Eqs.~\eqref{drift-event-inclusion} and~\eqref{upper-Z-restriction} that 
\begin{align*}
	Z_{\alpha,\beta,t}(F)&\ge\  \E\left[
 F\1_{E_s\cap S^R_{s,t}}
 e^{N_t^{(2n-1,\alpha,t)}-\frac12\langle N^{(2n-1,\alpha,t)}\rangle_t+Q_{s,t}^{(2n-1,\alpha)}}
 \right]
 \exp\left(
 -Ce^{-2n\left(2(1-\frac{1}{2n})
                  -\frac12(\gb^2+\eta)\right)s}
 \right)\\
 &=\Q(S^R_{s,t})e^{V_F(s,t)}\exp\left(
 -Ce^{-2n\left(2(1-\frac{1}{2n})
                  -\frac12(\gb^2+\eta)\right)s}
 \right),
\end{align*}
where in the equality we used the fact that $S^R_{s,t}$ is independent of $\mathcal{F}_s$ under $\Q$ and the martingale property of $\bigl(e^{N_u^{(2n-1,\alpha,t)}-\langle N^{(2n-1,\alpha,t)}\rangle_u/2}\bigr)_{0\le u\le t}$. Then Eq.~\eqref{gaussian-event-probability} and our parameter choices for $\eta,\gamma$ prove item 1.

Starting from Eq.~\eqref{gaussian-event-probability}, we make the observation that $(R_u-R_s)_{s\le u\le t}$ under $\Q$ has the same law as $(X_u-X_s)_{s\le u\le t}$ under $\P$. Therefore, it suffices to show the similar statement for $X$ under $\P$. The martingale property gives, for $s\le u\le v$,
\[
 |X_u(x)-X_u(y)|
 \le(e^v|x-y|)^\gamma
       \E[H_v^{(\gamma)}\mid\mathcal F_u].
\]
Multiplying both sides by $e^{-(\gamma+\eta)v}$ and
integrating $v$ over $[u,\infty)$, we obtain
\[
 |X_u(x)-X_u(y)|
 \le(\gamma+\eta)e^{(\gamma+\eta)u}|x-y|^\gamma
 \E\left[
 \int_s^\infty e^{-\eta v}H_v^{(\gamma)}\dd v
 \,\middle|\,\mathcal F_u
 \right].
\]
By Doob's inequality and Eq.~\eqref{Holder-moment-n},
applied with $\alpha=0$,
\[
 \P\left(
 \sup_{s\le u\le t}
 \E\left[
 \int_s^\infty e^{-\eta v}H_v^{(\gamma)}\dd v
 \,\middle|\,\mathcal F_u
 \right]
 \ge\frac{1}{2(\gamma+\eta)}
 \right)
 \le Ce^{-s(\eta+\gamma-1)}.
\]
Outside this event,
\[
 |X_u(x)-X_u(y)|
 \le\frac12e^{(\gamma+\eta)u}|x-y|^\gamma,
 \qquad s\le u\le t,\quad x,y\in \Lambda,
\]
and hence, simultaneously for these $u,x,y$,
\[
 |X_u(x)-X_u(y)-X_s(x)+X_s(y)|
 \le e^{\eta u}(e^u|x-y|)^\gamma.
\]
It follows that
\[
\P\bigg(\{|X_u(x)-X_u(y)-X_s(x)+X_s(y)|
 \le e^{\eta u}(e^u|x-y|)^\gamma,\ s\le u\le t,\ x,y\in \Lambda\}\bigg)\ge 1-Ce^{-s(\eta+\gamma-1)}
\]
proving Eq.~\eqref{gaussian-event-probability}. To prove Eq.~\eqref{drift-event-inclusion}, we first estimate $\dd\langle X(z)-X(y),N^{(2n-1,\alpha,t)}\rangle_u$. It\^o's formula gives
\[
 \dd\langle X(z),\cos(\gb Y_{p}(\x,\cdot))\rangle_u
 =-\gb\sin(\gb Y_{p}(\x,u))
 \sum_{a=1}^i\la_a^{(p)}Q_u(z,x_a)\dd u.
\]
Consequently,
\begin{multline}\label{firstd}
 \dd\langle X(z),N^{(2n-1,\alpha,t)}\rangle_u
 =-\gb\sum_{i=1}^{2n-1}\alpha^i
 \sum_{p=0}^{\lfloor i/2\rfloor}
 \int_{\Lambda^i}\mathcal F^{(i,p)}(\x,u,t)
 e^{\fb\sum_{a=1}^iK_u(x_a,x_a)}
 \sin(\gb Y_{p}(\x,u))\\
 \times\left(\sum_{a=1}^i\la_a^{(p)}Q_u(z,x_a)\right)
 \mu(\dd\x)\dd u.
\end{multline}
For $i\ge2$, using
Lemma~\ref{lem:cosine-test}, item~3,
Lemma~\ref{lem:energy-multiplier}, and
Proposition~\ref{prop:pairing} gives for each contribution
to the bracket density the upper bound
\[
 C(H_u^{(\gamma)})^i e^{\fb iu}e^{-2u}
 \int_u^t e^{-2(i-1)v}e^{\fb(i-1)(v-u)}\dd v
 \le C(H_u^{(\gamma)})^i e^{-ui(2-\fb)}.
\]
For $i=1$, since $\mathcal F^{(1,0)}=1$, the same conclusion
follows directly from $|\sin|\le1$ and Eq.~\eqref{laderive}
\[
 e^{\fb u}\sup_{z\in \Lambda}\int_\Lambda |Q_u(z,x)|\mu(\dd x)
 \le Ce^{-(2-\fb)u}.
\]
Therefore we have
\begin{equation*}
 \left|\dd\langle X(z),N^{(2n-1,\alpha,t)}\rangle_u\right|
 \le C\sum_{i=1}^{2n-1}(H_u^{(\gamma)})^i
 e^{-iu(2-\fb)}\dd u.
\end{equation*}
Now for $\dd\langle X(z)-X(y),N^{(2n-1,\alpha,t)}\rangle_u$, we use Eq.~\eqref{firstd}, apply the mean value theorem to the differences $Q_u(z,x_a)-Q_u(y,x_a)$, and invoke 
the derivative assertion in Lemma~\ref{lem:cosine-test} to obtain
\begin{equation}\label{upper-drift}
 \left|\dd\langle X(z)-X(y),N^{(2n-1,\alpha,t)}\rangle_u\right|
 \le C\sum_{i=1}^{2n-1}(H_u^{(\gamma)})^i
 e^{-iu(2-\fb)}(1\wedge e^u|z-y|)\dd u.
\end{equation}
For fixed $t\ge s$, write
\[
 A_u(x)\coloneqq\int_s^u\dd\langle X(x),N^{(2n-1,\alpha,t)}\rangle_v.
\]
By Eq.~\eqref{upper-drift}, for $s\le  u\le t$,

\[
 |A_u(x)-A_u(y)|\le C|x-y|^\gamma
 \int_s^u\sum_{i=1}^{2n-1}(H_v^{(\gamma)})^i
 e^{-i(2-\fb)v+\gamma v}\dd v.
\]
The integrand is almost surely integrable over $[s,t]$ by
Eq.~\eqref{Holder-moment-n} with $\alpha=0$ and Fubini's theorem, implying $A_u(x)-A_u(y)$ is continuous in $u$.
We may choose $t_0\in(s,t]$ such that
\[
C\int_s^{t_0}\sum_{i=1}^{2n-1}(H_r^{(\gamma)})^i
e^{-i(2-\fb)r+\gamma r}\dd r
\le e^{(\eta+\gamma)s}.
\]
Then for $s\le v\le t_0$ and $x,y\in \Lambda$,
\[
|A_v(x)-A_v(y)|
\le e^{\eta v}(e^v|x-y|)^\gamma.
\]
We let $t_0$ now denote the maximal time up to which this bound
holds for every $x,y\in \Lambda$, the preceding argument giving
$t_0>s$. By continuity of $A_v(x)-A_v(y)$ the bound is satisfied at $v=t_0$ as well. On $E_s\cap S^R_{s,t}$ and $s\le v\le t_0$ the decomposition
\[
\begin{aligned}
 X_v(x)-X_v(y)
 &=X_s(x)-X_s(y)+R_v(x)-R_v(y)-R_s(x)+R_s(y)
 +A_v(x)-A_v(y)
\end{aligned}
\]
gives $|X_v(x)-X_v(y)|\le3e^{\eta v}(e^v|x-y|)^\gamma$, implying $E_s\cap S^R_{s,t}\subset S_{s,t_0}$.
Repeating the proof of Eq.~\eqref{upper-drift} with this coefficient
in the estimates of Lemma~\ref{lem:cosine-test}, and using
$2-\fb-\eta>0$, gives
\[
\begin{aligned}
 |A_v(x)-A_v(y)|
 &\le C(e^v|x-y|)^\gamma
 \int_s^v e^{-(2-\fb-\eta)u}\dd u\\
 &\le Ce^{-(2-\fb-\eta)s}(e^v|x-y|)^\gamma\\
 &\le\tfrac12 e^{\eta v}(e^v|x-y|)^\gamma.
\end{aligned}
\]
This improves the coefficient in the original estimate from 1 to 1/2, so $t_0=t$.
\end{proof}
\begin{proof}[Proof of item 2]
We will first prove the lower bound in the case $\alpha\ge0$. Conditioning the stochastic exponential at time $s$ gives
\begin{equation*}
 \frac{e^{V_F(s,t)}}{Z_{\alpha,\beta,t}(F)}
 =\frac{\E_{\alpha,t}\left[
 F\1_{E_s}e^{Q_{s,t}^{(2n-1,\alpha)}-Q_{t,t}^{(2n-1,\alpha)}}
 \right]}{\E_{\alpha,t}[F]}.
\end{equation*}
Using $e^{-x}\ge1-x$ and $1\le F\le\|F\|_\infty$,
\begin{equation}\label{ratio-lower-bound}
\begin{aligned}
 \frac{e^{V_F(s,t)}}{Z_{\alpha,\beta,t}(F)}
 &\ge1-\|F\|_\infty\P_{\alpha,t}(E_s^c)-\|F\|_\infty\E_{\alpha,t}[|Q_{t,t}^{(2n-1,\alpha)}-Q_{s,t}^{(2n-1,\alpha)}|].
\end{aligned}
\end{equation}
By Markov's inequality and Eq.~\eqref{Holder-moment-n},
\[
 \P_{\alpha,t}(E_s^c)
 \le e^{-\eta s}\E_{\alpha,t}[H_s^{(\gamma)}]
 \le Ce^{-s(\eta-1+\gamma)}.
\]
Taking the expectation in Eq.~\eqref{remainder-integral-bound} with
respect to $\E_{\alpha,t}$, after using Lemma~\ref{lem:cosine-test}, item~1, and using Eq.~\eqref{Holder-moment-n}
gives
\[
\E_{\alpha,t}[|Q_{t,t}^{(2n-1,\alpha)}-Q_{s,t}^{(2n-1,\alpha)}|]
 \le Ce^{-[2(2n-1)-n(\gb^2+1-\gamma)]s}.
\]
The exponent is positive by our parameter choice. Hence
\begin{equation*}
 \frac{e^{V_F(s,t)}}{Z_{\alpha,\beta,t}(F)}\ge1-o_s(1)
\end{equation*}
uniformly in $t\ge s$. Then let $\alpha<0$. Item 1 of this Lemma and Lemma~\ref{lem:fixed-scale-convergence} give $\inf_{t\ge s}Z_{\alpha,\beta,t}(1)>0$ for sufficiently large fixed $s$.
By the case $\alpha\ge0$ proved above and Theorem~\ref{thm:cumulants},
\[
 \frac{\E[e^{|\alpha|M_t}]}{\E[e^{\alpha M_t}]}
 =\frac{\E[e^{|\alpha|M_t}]
 e^{-\sum_{i=1}^{2n-1}\frac{|\alpha|^i}{i!}\mathcal C_i(M_t)}}
 {Z_{\alpha,\beta,t}(1)}
 \exp\left(2\sum_{\substack{1\le i\le2n-1\\ i\text{ odd}}}
 \frac{|\alpha|^i}{i!}\mathcal C_i(M_t)\right)
 \le C,\qquad t\ge s.
\]
Dividing Eq.~\eqref{Holder-moment-unnormalized} by
$\E[e^{\alpha M_t}]$ and using this bound gives
Eq.~\eqref{Holder-moment-n} also for $\alpha<0$,
with a constant depending on $\alpha$.
Thus the estimates following Eq.~\eqref{ratio-lower-bound}
apply unchanged.
\end{proof}

\subsection*{Statements and Declarations}

\subsubsection*{Funding}
This work was supported by the Swiss National Science Foundation (SNSF) grant 207410, ``2d constructive field theory with exponential interactions''.

\subsubsection*{Competing interests}
The authors have no competing interests to declare that are relevant to the content of this article.

\subsubsection*{Data availability}
Data sharing not applicable to this article as no datasets were generated or analysed during the current study.

\appendix
\section{Algebra of disjoint coordinate swaps}\label{app:swap-algebra}

Let $M\in\Omega_n$. For $A\subseteq M$, use the notation $D_A$ and $\tau_A$ introduced before Definition~\ref{TestDef}.
The swaps in $M$ are disjoint, so they commute. In the following, all identities are pointwise and we do not assume any regularity for functions $f,g$.

\begin{lemma}\label{lem:swap-subsets}
For $A\subseteq M$ and any function $f$,
\begin{align}
 D_M f&=\sum_{B\subseteq M}(-1)^{|B|}\tau_B f,\notag
 \\
 D_M f&=\sum_{B\subseteq M\setminus A}
 (-1)^{|B|}\tau_BD_Af.
 \label{app:partial-subset-expansion}
\end{align}
\end{lemma}
\begin{proof}
If $M=\{m\}$ we simply write $D_m=\tau_\emptyset-\tau_m$. By induction, if $|M|>1$, we rewrite 
\[
D_Mf=D_mD_{M\backslash\{m\}}f =\sum_{B\subseteq M\backslash\{m\}}(-1)^{|B|}\tau_B f-\sum_{B\subseteq M\backslash\{m\}}(-1)^{|B|}\tau_{B\cup\{m\}}f,
\]
which gives the first identity. For the second, write $D_M=D_{M\setminus A}D_A$ and expand only the first factor.
\end{proof}

\begin{lemma}\label{lem:swap-product-rules}
For any functions $f,g$,
\begin{align}
 D_M(fg)
 &=\sum_{A\subseteq M}(D_Af)
       (\tau_A D_{M\setminus A}g),
 \label{app:product-rule-first}\\
 &=\sum_{A\subseteq M}(\tau_{M\setminus A}D_Af)
       (D_{M\setminus A}g).
 \label{app:product-rule-second}
\end{align}
\end{lemma}
\begin{proof}
For $M=\{m\}$ we get the first form
\begin{equation}\label{app:one-swap-product-rule}
 D_m(fg)
 =fD_mg+(D_mf)\tau_mg.
\end{equation}
With the aim of inducting on $|M|$, we rewrite $D_M(fg)=D_mD_{M'}(fg)$, where $M'=M\setminus\{m\}$. Applying Eq.~\eqref{app:one-swap-product-rule} to each term in the expanded form of $D_{M'}(fg)$, we obtain 
\begin{align*}
\sum_{A\subseteq M'}\bigg[(D_Af)(\tau_AD_{M\setminus A}g)+(D_{A\cup\{m\}}f)(\tau_{A\cup\{m\}}D_{M'\setminus A}g)\bigg]
\end{align*}
proving Eq.~\eqref{app:product-rule-first}. Eq.~\eqref{app:product-rule-second} is obtained by swapping $f$ and $g$ in Eq.~\eqref{app:product-rule-first} and summing over sets $M\setminus A$.
\end{proof}
\section{Estimates for \texorpdfstring{$\mathcal Q$}{Q}}\label{app:Q-estimates}
\begin{proof}[Proof of Lemma~\ref{lem:Q-estimates}]
Using Eq.~\eqref{laderive}, and choosing $C$ sufficiently large for the estimates below, we obtain Eq.~\eqref{Q-basic}.
If $i=2p$, we can re-arrange the sum such that
\begin{equation*}
	\mathcal{Q}_{i,j,p,q}(\x_1,\x_2,u)=\sum_{k=1}^{p}\sum^{i+j}_{l=i+1}\la_{l-i}^{(q)}(Q_u(x_{p+\sigma_1(k)},x_l)-Q_u(x_k,x_l))
\end{equation*}
where $\sigma_1\in S_p$ is an optimizing permutation for $m(\x_1)$. Then, using Eq.~\eqref{laderive} and the mean value theorem,
\begin{equation}\label{Q-one-difference}
 |Q_u(x_{p+\sigma_1(k)},x_l)-Q_u(x_k,x_l)|\le C m(\x_1)e^u\sup_{c\in[0,1]}\frac{1}{(1+e^u|c x_k+(1-c)x_{p+\sigma_1(k)}-x_l|)^{3}}.
\end{equation}
When $m(\x_1)< e^{-u}$, the fraction in Eq.~\eqref{Q-one-difference} is of the same order as $(1+e^u|x_k-x_l|)^{-3}$. Combining Eqs.~\eqref{Q-basic},\eqref{Q-one-difference} proves Eq.~\eqref{Q-one-neutral}.
A similar bound holds when $j=2q$. When both $i=2p$ and $j=2q$, let $\sigma_2\in S_q$ be an optimizing permutation for $m(\x_2)$. We write
\begin{multline*}
\mathcal{Q}_{i,j,p,q}(\x_1,\x_2,u)=\sum_{k=1}^{p}\sum^{i+q}_{l=i+1}
\bigl(Q_u(x_{p+\sigma_1(k)},x_{i+q+\sigma_2(l-i)})-Q_u(x_k,x_{i+q+\sigma_2(l-i)})\\
-Q_u(x_{p+\sigma_1(k)},x_l)+Q_u(x_k,x_l)\bigr).
\end{multline*}
Then Eq.~\eqref{laderive} and the mean value theorem yield
\begin{multline}\label{Q-two-differences}
 |Q_u(x_{p+\sigma_1(k)},x_{i+q+\sigma_2(l-i)})-Q_u(x_k,x_{i+q+\sigma_2(l-i)})-Q_u(x_{p+\sigma_1(k)},x_l)+Q_u(x_k,x_l)|\\
 \le C m(\x_1)e^u m(\x_2)e^u\sup_{c_1,c_2\in[0,1]}\frac{1}{(1+e^u|(c_1 x_k+(1-c_1)x_{p+\sigma_1(k)})-(c_2 x_l+(1-c_2)x_{i+q+\sigma_2(l-i)})|)^{3}}.
\end{multline}
When $m(\x_1)<e^{-u}$ and $m(\x_2)<e^{-u}$, the fraction in Eq.~\eqref{Q-two-differences} is of order $(1+e^u|x_k-x_l|)^{-3}$. Combining this with Eqs.~\eqref{Q-basic} and~\eqref{Q-one-neutral} (and the $j=2q$ version of~\eqref{Q-one-neutral}) proves Eq.~\eqref{Q-two-neutral}.
\end{proof}


\begin{thebibliography}{99}

\bibitem{BauerschmidtHofstetter2022}
R.~Bauerschmidt and M.~Hofstetter,
\emph{Maximum and coupling of the sine-Gordon field},
The Annals of Probability \textbf{50} (2022), 455--508.
\href{https://doi.org/10.1214/21-AOP1537}{doi:10.1214/21-AOP1537}.

\bibitem{BauerschmidtWebb2024}
R.~Bauerschmidt and C.~Webb,
\emph{The Coleman correspondence at the free fermion point},
Journal of the European Mathematical Society \textbf{26} (2024), 3137--3241.
\href{https://doi.org/10.4171/JEMS/1329}{doi:10.4171/JEMS/1329}.

\bibitem{BenfattoGallavottiNicolo1982}
G.~Benfatto, G.~Gallavotti, and F.~Nicol\`o,
\emph{On the massive sine-Gordon equation in the first few regions of collapse},
Communications in Mathematical Physics \textbf{83} (1982), 387--410.
\href{https://doi.org/10.1007/BF01213609}{doi:10.1007/BF01213609}.

\bibitem{BringmannCao2026}
B.~Bringmann and S.~Cao,
\emph{Global well-posedness of the dynamical sine-Gordon model up to $6\pi$},
The Annals of Probability \textbf{54} (2026), 1564--1607.
\href{https://doi.org/10.1214/25-AOP1797}{doi:10.1214/25-AOP1797}.

\bibitem{BrydgesKennedy1987}
D.~C.~Brydges and T.~Kennedy,
\emph{Mayer expansions and the Hamilton--Jacobi equation},
Journal of Statistical Physics \textbf{48} (1987), 19--49.
\href{https://doi.org/10.1007/BF01010398}{doi:10.1007/BF01010398}.

\bibitem{ChandraFeltesWeber2026}
A.~Chandra, G.~de~Lima~Feltes, and H.~Weber,
\emph{A priori bounds for the generalised parabolic Anderson model},
Communications on Pure and Applied Mathematics \textbf{79} (2026), 1315--1394.
\href{https://doi.org/10.1002/cpa.70025}{doi:10.1002/cpa.70025}.

\bibitem{ChandraHairerShen2018}
A.~Chandra, M.~Hairer, and H.~Shen,
\emph{The dynamical sine-Gordon model in the full subcritical regime},
\href{https://arxiv.org/abs/1808.02594}{arXiv:1808.02594} (2018).

\bibitem{DimockHurd1993}
J.~Dimock and T.~R.~Hurd,
\emph{Construction of the two-dimensional sine-Gordon model for $\beta<8\pi$},
Communications in Mathematical Physics \textbf{156} (1993), 547--580.
\href{https://doi.org/10.1007/BF02096863}{doi:10.1007/BF02096863}.

\bibitem{DimockHurd2000}
J.~Dimock and T.~R.~Hurd,
\emph{Sine-Gordon revisited},
Annales Henri Poincar\'e \textbf{1} (2000), 499--541.
\href{https://doi.org/10.1007/s000230050005}{doi:10.1007/s000230050005}.

\bibitem{Frohlich1975}
J.~Fr\"ohlich,
\emph{Quantized ``sine-Gordon'' equation with a nonvanishing mass term
in two space-time dimensions},
Physical Review Letters \textbf{34} (1975), 833--836.
\href{https://doi.org/10.1103/PhysRevLett.34.833}{doi:10.1103/PhysRevLett.34.833}.

\bibitem{GubinelliHairerOhZine2025}
M.~Gubinelli, M.~Hairer, T.~Oh, and Y.~Zine,
\emph{A simple construction of the sine-Gordon model via stochastic quantization},
Journal of the London Mathematical Society \textbf{112} (2025), e70214.
\href{https://doi.org/10.1112/jlms.70214}{doi:10.1112/jlms.70214}.

\bibitem{GubinelliMeyer2024}
M.~Gubinelli and S.-J.~Meyer,
\emph{The FBSDE approach to sine-Gordon up to $6\pi$},
\href{https://arxiv.org/abs/2401.13648v3}{arXiv:2401.13648v3}
(2024, revised 2026).

\bibitem{HairerShen2016}
M.~Hairer and H.~Shen,
\emph{The dynamical sine-Gordon model},
Communications in Mathematical Physics \textbf{341} (2016), 933--989.
\href{https://doi.org/10.1007/s00220-015-2525-3}{doi:10.1007/s00220-015-2525-3}.

\bibitem{JunnilaSaksmanWebb2020}
J.~Junnila, E.~Saksman, and C.~Webb,
\emph{Imaginary multiplicative chaos: moments, regularity and
connections to the Ising model},
The Annals of Applied Probability \textbf{30} (2020), 2099--2164.
\href{https://doi.org/10.1214/19-AAP1553}{doi:10.1214/19-AAP1553}.

\bibitem{KosterlitzThouless1973}
J.~M.~Kosterlitz and D.~J.~Thouless,
\emph{Ordering, metastability and phase transitions in two-dimensional systems},
Journal of Physics C: Solid State Physics \textbf{6} (1973), 1181--1203.
\href{https://doi.org/10.1088/0022-3719/6/7/010}{doi:10.1088/0022-3719/6/7/010}.

\bibitem{Kunita2019}
H.~Kunita,
\emph{Stochastic flows and jump-diffusions},
Probability Theory and Stochastic Modelling \textbf{92},
Springer, Singapore, 2019.
Section~1.8, Lemmas~1.8.1--1.8.2, pp.~42--44.
\href{https://doi.org/10.1007/978-981-13-3801-4}{doi:10.1007/978-981-13-3801-4}.

\bibitem{LacoinRhodesVargas2023}
H.~Lacoin, R.~Rhodes, and V.~Vargas,
\emph{A probabilistic approach of ultraviolet renormalization in the boundary
sine-Gordon model},
Probability Theory and Related Fields \textbf{185} (2023), 1--40.
\href{https://doi.org/10.1007/s00440-022-01174-5}{doi:10.1007/s00440-022-01174-5}.
Preprint: \href{https://arxiv.org/abs/1903.01394v2}{arXiv:1903.01394v2} (2020).
Numbered references to this work follow this arXiv version.

\bibitem{Meyer2026}
S.-J.~Meyer,
\emph{An FBSDE construction of the sine-Gordon EQFT for
$\beta^2<\frac67\,8\pi$ and perturbative renormalization in the full subcritical regime},
\href{https://arxiv.org/abs/2607.20632}{arXiv:2607.20632} (2026).

\bibitem{Nicolo1983}
F.~Nicol\`o,
\emph{On the massive sine-Gordon equation in the higher regions of collapse},
Communications in Mathematical Physics \textbf{88} (1983), 581--600.
\href{https://doi.org/10.1007/BF01211960}{doi:10.1007/BF01211960}.

\bibitem{NicoloPerfetti1989}
F.~Nicol\`o and P.~Perfetti,
\emph{The sine-Gordon field theory model at $\alpha^2=8\pi$,
the non-superrenormalizable theory},
Communications in Mathematical Physics \textbf{123} (1989), 425--452.
\href{https://doi.org/10.1007/BF01238809}{doi:10.1007/BF01238809}.

\bibitem{NicoloRennSteinmann1986}
F.~Nicol\`o, J.~Renn, and A.~Steinmann,
\emph{On the massive sine-Gordon equation in all regions of collapse},
Communications in Mathematical Physics \textbf{105} (1986), 291--326.
\href{https://doi.org/10.1007/BF01211104}{doi:10.1007/BF01211104}.

\bibitem{Pelaic2025}
J.~Pelai\v{c},
\emph{The subcritical finite-volume massive sine-Gordon model},
\href{https://arxiv.org/abs/2508.13778v2}{arXiv:2508.13778v2} (2025).

\bibitem{TangXu2025}
S.~Tang and R.~Xu,
\emph{A quadratic BSDE approach to normalization for the finite volume
2D sine-Gordon model in the finite ultraviolet regime},
\href{https://arxiv.org/abs/2501.12172v2}{arXiv:2501.12172v2} (2025).

\bibitem{Vihko2026}
S.~Vihko,
\emph{The massless sine-Gordon model with logarithmic correlations
in arbitrary dimension},
\href{https://arxiv.org/abs/2606.06657v2}{arXiv:2606.06657v2} (2026).

\end{thebibliography}
\end{document}